\documentclass[11pt]{amsart}
\usepackage{amssymb, amsmath, amsthm, mathrsfs, amsfonts, latexsym,color}
\usepackage[margin=1in]{geometry}

\newtheorem{theorem}{Theorem}[section]
\newtheorem{lemma}[theorem]{Lemma}

\newtheorem{proposition}[theorem]{Proposition}

\newtheorem{remark}[theorem]{Remark}

\theoremstyle{definition}

\newtheorem{assumption}[theorem]{Assumption}

\numberwithin{equation}{section}

\renewcommand{\P}{\mathrm{P}}

\newcommand{\E}{\mathrm{E}}
\newcommand{\R}{\mathbb{R}}

\newcommand{\1}{\mathbf{1}}

\begin{document}

\title{Multiple Points of Gaussian Random Fields}

\author{Robert C.~Dalang, Cheuk Yin Lee, Carl Mueller, and Yimin Xiao}

\address[Robert C.~Dalang]{Institut de math\'ematiques,
\'Ecole Polytechnique F\'ed\'erale de Lausanne, Station 8,
CH-1015 Lausanne, Switzerland}
\email{robert.dalang@epfl.ch}

\address[Cheuk Yin Lee]{Department of Statistics and Probability,
Michigan State University, East Lansing, MI 48824, United States.\newline
Current address: Institut de math\'ematiques,
\'Ecole Polytechnique F\'ed\'erale de Lausanne, Station 8,
CH-1015 Lausanne, Switzerland}
\email{cheuk.lee@epfl.edu}

\address[Carl Mueller]{Department of Mathematics, University of
Rochester, Rochester, NY  14627,
United States}
\email{carl.e.mueller@rochester.edu}

\address[Yimin Xiao]{Department of Statistics and Probability,
Michigan State University, East Lansing, MI 48824, United States}
\email{xiao@stt.msu.edu}



\thanks{The research of R.C. Dalang is partially supported by the Swiss National Foundation for Scientific Research,
the research of C. Mueller is partially supported by a Simons grant, and Y. Xiao is partially supported by NSF
grants DMS-1607089 and DMS-1855185.}


\begin{abstract}
This paper is concerned with the existence of multiple points of Gaussian random fields. 
Under the framework of Dalang et al. (2017), we prove that, for a wide class of Gaussian random fields, multiple points do not exist in
critical dimensions. The result is applicable to fractional Brownian sheets and the solutions of systems of stochastic heat and wave equations.
\end{abstract}

\maketitle

\section{Introduction}

Let $v = \{ v(x), x \in \mathbb{R}^k \}$ be a centered continuous $\mathbb{R}^d$-valued
Gaussian random field defined on a probability space $(\Omega, \mathscr{F}, \P)$ with i.i.d.\ components.
Write $v(x) = (v_1(x), \dots, v_d(x))$ for $x \in \mathbb{R}^k$. For a set $T \subset \mathbb{R}^k $
(e.g., $T = (0, \infty)^k$, or $T = [0, 1]^k$)
and  an integer $m \ge 2$, we say that $z \in \mathbb{R}^d$ is an \emph{$m$-multiple point}
of $v(x)$ on $T$ if, with positive probability,
there are $m$ distinct points $x^1, \dots, x^m \in T$ such that $z = v(x^1) = \dots = v(x^m)$.



Several authors have studied the existence of multiple points of Gaussian random fields.
Sufficient conditions or necessary  conditions for the case of a fractional Brownian motion
$B^H = \{B^H(t), t \in \R^k\}$ in $\R^d$ were proved by K\^ono \cite{Kono}, Goldman \cite{G81},
Rosen \cite{Rosen84}. Their results show that if $km > (m-1)Hd$ then $B^H$ has $m$-multiple
points on any interval $T \subseteq \R^k$; and if $km < (m-1)Hd$ then $B^H$ has no $m$-multiple points
on $\R^k\backslash \{0\}$. Rosen \cite{Rosen84} also considered the existence of multiple
points of the Brownian sheet by studying its self-intersection local times.

In the critical dimensions (i.e., $km = (m-1)Hd$ for $B^H$), the problem for proving the
non-existence of multiple points is more difficult. For fractional Brownian motion and the
Brownian sheet, the problem was resolved by Talagrand \cite{T98} and by
Dalang et al.~\cite{dknwx12} and Dalang and Mueller  \cite{dm14}, respectively. The methods in
\cite{T98} and \cite{dknwx12,dm14} are different.

Our research in this paper is motivated by the interest in studying
the intersection problems for the solutions of systems of stochastic heat and wave
equations with constant coefficients, where the method in \cite{dknwx12,dm14} fails in general.
Our main purpose is to continue the work of \cite{DMX17} and extend Talagrand's approach in \cite{T98}
to a large class of Gaussian random fields which include fractional Brownian sheets and the solutions
of systems of stochastic heat and wave equations with constant coefficients. As a byproduct, our
theorem provides an alternative proof for the results in \cite{dknwx12,dm14} by using general
Gaussian principles and the harmonizable representation of the Brownian sheet.

The result of this paper relies on a covering argument originated by Talagrand \cite{T98}, and later further 
developed by Dalang et al.~\cite{DMX17}. In this paper, the main ingredient for the covering argument is 
Proposition \ref{prop}, which states that for any distinct points $s^1, \dots, s^m \in T$, with high probability, 
there are small neighbourhoods of $s^i$ in which the maximum of the increments $v(x^i) - v(s^i)$ ($1 \le i \le m$) 
could be smaller than one would expect from the H\"older regularity. This observation allows us to use balls 
of different radii to construct an efficient random cover for the set of multiple points, which is essential for 
proving the non-existence of multiple points in the critical dimension.

The paper is organized as follows. In Section 2, we state our assumptions and the
main result of this paper, Theorem \ref{main thm}. In Section 3, we establish some
necessary ingredients for proving Theorem \ref{main thm} and, in Section 4, we prove
the main theorem. In Section 5, we provide several examples of Gaussian random fields
to which the theorem can be applied. These examples include the Brownian sheet,
fractional Brownian sheets, and the solutions of systems of stochastic heat and wave equations.

Throughout the article, we use $K$ or $c$ to denote a constant that may vary at each occurrence.
Specific constants will be denoted by $K_1, K_2$, $c_1$, etc.

\section{Assumptions and the main result}

By a compact interval (or rectangle) in $\mathbb{R}^k$ we mean a set $I$ of the form
$\prod_{j=1}^k [c_j, d_j]$, where $c_j < d_j$. Throughout this paper, we assume that
$T \subset \mathbb{R}^k $ is a fixed index set that can be written as a countable
union of compact intervals. To avoid triviality in studying the multiple points of
$\{ v(x), x \in \mathbb{R}^k \}$, one may take, for example, $T = \mathbb{R}^k
\backslash\{0\}$ or $T = (0, \infty)^k$.


In the following, Assumption \ref{a1} is a slightly simplified reformulation of 
Assumption 2.1 in \cite{DMX17} and 
Assumption \ref{a2} below is a reinforced version of Assumption 2.4 in \cite{DMX17}.

\begin{assumption}\label{a1}
There exists a centered Gaussian random field $\{ v(A, x) , A \in
\mathscr{B}(\mathbb{R}_+), x \in T \}$, where $\mathscr{B}(\mathbb{R}_+)$ is the
Borel $\sigma$-algebra on $\mathbb{R}_+ =[0, \infty)$, such that the following hold:

(a) For all $x \in T$, $A \mapsto v(A, x)$ is an $\mathbb{R}^d$-valued
white noise (or, more generally, an independently scattered Gaussian noise
with a control measure $\mu$) with i.i.d.\ components, $v(\mathbb{R}_+, x)
= v(x)$, and $v(A, \cdot)$ and $v(B, \cdot)$ are independent whenever $A$
and $B$ are disjoint.

(b) There exist constants $\gamma_j > 0$, $j = 1, \dots, k$ with the following
properties: For every compact interval $F \subset T$, there exist constants
$c_0 > 0$ and $ a_0 \ge 0$ such that for all $a_0 \le a \le b \le \infty$
and $x, y \in F$,
\begin{align}
\| v([a, b), x) - v(x) - v([a, b), y) + v(y) \|_{L^2}
\le c_0 \bigg( \sum_{j=1}^k a^{\gamma_j}|x_j - y_j| + b^{-1} \bigg),
\end{align}
 and
\begin{align}
\| v([0, a_0), x) - v([0, a_0), y) \|_{L^2} \le c_0 \sum_{j=1}^k |x_j - y_j|.
\end{align}
In the above, $\|X\|_{L^2} = \big(\E|X|^2\big)^{1/2}$ for a random vector $X$, 
where $|X|$ is the Euclidean norm of $X$.
\end{assumption}

Notice that in Assumption \ref{a1} the constants  $a_0$ and $c_0$ may depend on $F$,
but $\gamma_j$ ($ j = 1, \ldots, k$) do not.
As shown by Dalang et al. \cite{DMX17}, the parameters $\gamma_j$ ($ j = 1, \ldots, k$)
play important roles in characterize sample path properties (e.g., regularity, fractal
properties, hitting probabilities) of the random field $\{ v(x),\, x \in T \}$.

Let $\alpha_j = (\gamma_j + 1)^{-1}$ and $Q = \sum_{j = 1}^k \alpha_j^{-1}$.
Define the metric $\Delta$ on $\mathbb{R}^k$ by
\begin{equation}\label{Def:De}
\Delta(x, y) = \sum_{j=1}^k |x_j - y_j|^{\alpha_j}.
\end{equation}

For $x \in T$ and $r > 0$,  denote by $S(x, r) = \{ y \in \mathbb{R}^k : \Delta(x, y) \le r \}$
the closed ball with center $x$ and radius $r$ in the metric $\Delta$ in (\ref{Def:De}) and
let $B_r(x) = \prod_{j=1}^k [x_j - r^{1/\alpha_j}, \, x_j + r^{1/\alpha_j}]$.
Notice that  $S(x, r) \subseteq B_r(x)$ and $B_{r/k}(x) \subseteq S(x, r)$.

\begin{assumption}\label{a2}
For every compact interval $F \subset T$,  there exist
constants $0 < \varepsilon_0 \le 1$, $c \ge 3$, $c' \ge 2c$ 
and $\delta_j \in (\alpha_j, 1]$, $j = 1, \dots, k$, such that 
for every $0 < \rho \le \varepsilon_0$ there is a 
finite constant $C$ (which may depend on $F$, $\rho$ and $\delta_j$) 
and the following property holds:
%

For every $x \in F$ with $B_{c\rho}(x) \subset F$, there is $x' \in B_{c\rho}(x)$ 
such that for every $i = 1, \dots, d$ the condition 
 \begin{align}\label{a2:ineq}
\big|\E((v_i(y) - v_i(\bar{y}))v_i(x')) \big| \le
C \sum_{j=1}^k |y_j - \bar{y}_j|^{\delta_j}
\end{align}
holds for all $y, \bar y$ in each of the following cases:
\begin{enumerate}
\item[(i)] $y, \bar{y} \in B_{2\rho}(x)$;
\item[(ii)] $y, \bar{y} \in B_{2\rho}(\tilde x)$, for any $\tilde x$ such that
$B_{2\rho}(\tilde x) \subset F$ and $\Delta(x, \tilde x) \ge c'\rho$.
\end{enumerate}
\end{assumption}

Assumption \ref{a2} states that for every fixed point $x \in T$, 
one can find a reference point $x'$ 
(which may depend on $x$) such that for all $y, \bar{y}$ 
that are either in a small neighborhood of $x$, 
or are sufficiently away from a neighborhood of $x$ (hence away from $x'$),
the covariances in \eqref{a2:ineq}  are smoother than what one gets from 
the Cauchy--Schwarz inequality and 
Lemma \ref{lemma1} below, which lead to $\alpha_j$ instead of $\delta_j$ 
in the exponents. 

As an illustration example, let $\{ v(x),\, x \in \R^k\}$ be a real-valued fractional 
Brownian motion with index $H \in (0, 1)$ and let $F \subset \R^k \backslash\{0\}$ be a compact 
interval. Then for any $\rho > 0$ and $x',$ $y, \, \bar{y} \in F$ that satisfy $\min\{|y|, |\bar{y} |, 
|y - x' |, |\bar{y}-x' | \} \ge \rho$, where $|\cdot|$ is the Euclidean norm in $\R^k$, we use the 
mean value theorem to derive
\[
\begin{split}
\big|\E((v(y) - v(\bar{y}))v_i(x')) \big| 
&= \frac 1 2\Big(|y|^{2H} - |\bar{y} |^{2H} - |y - x' |^{2H} + |\bar{y}-x' |^{2H} \Big)\\
&\le C\, |y - \bar{y}|,
\end{split}
\]
where $C$ is a constant depending on $H$, $\rho$ (if $H \le 1/2$) and the compactness of $F$ (if $H > 1/2$). 
Hence Assumption \ref{a2} holds with $c = 3, c' = 6$ and 
$\delta_j = 1$ for $j = 1, \dots, k$. This was observed in Lemma 3.2 of Talagrand \cite{T98}. As we will see 
in Section 5, verifications of Assumption \ref{a2} for fractional Brownian sheets and the solutions to 
stochastic heat and wave equations are more involved. 


Now we introduce an additional non-degeneracy assumption.

\begin{assumption}\label{a3}
For any $m$ distinct points $x^1, \dots, x^m \in T$, $v_1(x^1), \dots, v_1(x^m)$
are linearly independent random variables, or equivalently, the Gaussian vector
$(v_1(x^1), \dots, v_1(x^m))$ is non-degenerate.
\end{assumption}

\begin{remark}
{\rm Assumption \ref{a3} is also equivalent to Var$(v_1(x^1)) > 0$ and, for
every $\ell = 2, \dots, m$, the
conditional variance of $v_1(x^\ell)$  given $ v_1(x^j),\, j \le \ell-1$ is positive.
}
\end{remark}

The main result of this paper is the following.

\begin{theorem}\label{main thm}
Let $m\geq 2$.
Suppose that Assumptions \ref{a1}, \ref{a2} and \ref{a3} hold.
If $mQ \le (m-1)d$, then $\{v(x), x \in T\}$ has no $m$-multiple points
almost surely.
\end{theorem}

\section{Preliminaries}

In this section, we provide some preliminaries that will be used for proving
Theorem \ref{main thm}. Clearly it suffices to prove that if $mQ \le (m-1)d$
then,  for every compact interval $F \subset T$, $\{v(x), x \in F\}$ has no $m$-multiple
points almost surely. Hence, without loss of generality, we assume in Sections
3 and 4 that $T$ is a compact interval.


Fix an integer $m \geq 2$. For an integer $n \ge 1$, let $A_n$ be the countable subset of 
$T^m$ defined by
\begin{equation}\label{Eq:An}
A_n = \{ (t^1, \dots, t^m) : t^i \in T \cap \mathbb{Q}^k, \Delta(t^i, t^j) \ge 1/n
\text{ for } i \ne j \}.
\end{equation}
Since $T$ is compact, we see that the closure $\overline{A}_n$ is compact. Moreover,  
$\bigcup_{n=1}^\infty A_n $ is dense in $T^m$.

For any constant  $\rho >0$,
let $B^i_\rho = B_\rho (t^i)$ ($ i = 1, \dots, m$). It is clear that if   $(t^1, \dots, t^m)  \in A_n$ and 
$\rho \in (0, 1/n)$ is small enough, then the intervals $B^i_\rho$  ($ i = 1, \dots, m$) are disjoint.

Consider the random set
\begin{align}\label{3.1}
M_{t^1, \dots, t^m; \rho} = \bigg\{ z \in \mathbb{R}^d : \exists \, (x^1, \dots, x^m)
\in \prod_{i=1}^m B^i_\rho \,
\text{ such that } z = v(x^1) = \dots = v(x^m) \bigg\},
\end{align}
which is the intersection of the images  $v(B^i_\rho)$ for $i= 1, \dots, m$ and is contained in 
the set of $m$-multiple points. On the other hand, 
let $z \in \R^d$ be an $m$-multiple point, i.e., there are $m$ distinct points $ x^1, \dots, x^m \in T$ 
such that $z = v(x^1) = \dots = v(x^m)$. Then for $n$ large enough we have $\Delta(x^i, x^j) > 3/n$ for 
$i \ne j$. It follows from the triangle inequality that, for any constant $\rho \in (0, 1/n)$, there is 
$(t^1, \dots, t^m) \in A_n$ such that $ \Delta(x^i, t^i) < \rho$. This implies that $z \in M_{t^1, \dots, t^m; \rho}$.
Thus we have verified that the set of $m$-multiple points of $\{ v(x) : x \in T \}$ 
can be written as a countable union
\begin{align}\label{3.2}
\bigcup_{n \ge 1} \, \bigcap_{\rho \in (0, 1/n)\cap \mathbb{Q}} \, \bigcup_{(t^1, \dots, t^m) \in A_n}\  
 M_{t^1, \dots, t^m; \rho}.
\end{align}
We will use this observation in Section 4 to complete the proof of  Theorem \ref{main thm}.

For the rest of this section, we fix an integer $n \ge 1$. Let $\rho_0 \in (0, 1/n)$ be the small constant given 
in Lemma \ref{lemma5} below. Notice that the constant $\rho_0$
is independent of $(t^1, \dots, t^m) \in A_n$. For simplicity of notation, we
assume that $B_{\rho_0} (t^i) \subseteq T$ for $ i = 1, \dots, m$ (otherwise we
take the intersection with $T$), and we omit the subscripts $t^1, \dots, t^m$ in \eqref{3.1} and write 
$M_{t^1, \dots, t^m; \rho},$ as $M_\rho$.

Recall from \cite{DMX17} that, under Assumption \ref{a1}, $\Delta$ provides an upper bound
for the $L^2$-norm of the increments of $\{ v(x), x \in T \}$ and in particular $v(x)$ is
continuous in $L^2(\Omega,  \mathscr{F}, \P)$.

\begin{lemma}\label{lemma1}\cite[Proposition 2.2]{DMX17}
Under Assumption \ref{a1}, for all $x, y \in T$ with $\Delta(x, y) \le \min\{a_0^{-1}, 1\}$, we have
$\|v(x) - v(y)\|_{L^2} \le 4c_0 \Delta(x, y)$.
\end{lemma}


Assumption \ref{a1} suggests that for any $s \in T$ and $x$ that is close to $s$, the increment
$v(x) - v(s)$  can be approximated well by
$v([a,b),x) - v([a,b), s)$ if we choose $a$ and $b$ carefully. The following lemma from
\cite{DMX17} quantifies the approximation error on $S(s,\, cr)$.

\begin{lemma} \label{lem3.2}
Let $c > 0$ be a constant. Consider $b > a > 1$, $\varepsilon_1 > r > 0$, where $\varepsilon_1 > 0$ 
is a small constant. Set
$$
   A =  \sum_{j=1}^k a^{\alpha_j^{-1} -1}\, r^{\alpha_j^{-1}} + b^{-1}.
$$
There are constants $A_0$, $\tilde K$ and $\tilde c$ (depending on
$c_0$ in Assumption \ref{a1} and $c$) such that if $A \leq A_0 r^{}$ and
\begin{equation}\label{b1}
   u \geq \tilde K A \log^{1/2}\left(\frac{r^{}}{A} \right) ,
\end{equation}
then for any $s \in T$,
\begin{align*}
\P\bigg\{\sup_{x \in S(s,\, c r)} \vert v(x) - v(s) - (v([a,b),x) -
v([a,b),s))\vert \geq u \bigg\}
\leq \exp\left(-\frac{u^2}{\tilde c A^2} \right).
\end{align*}
\end{lemma}

\begin{remark}
{\rm The constant $c$ in Lemma \ref{lem3.2} and Proposition \ref{prop} below
is not important. It merely helps to simplify the presentation in Section \ref{Sec:4},
where sometimes we switch back and forth between a ball $S(s, r)$ and an interval
$B_r(x)$.}
\end{remark}

For describing the contribution of the main part $v([a,b],x) - v([a,b], s)$, we will
apply the small ball probability estimate given in Lemma \ref{lemma2} below. We refer
to Lemma 2.2 of \cite{T95} for a general lower bound on the small ball probability
of Gaussian processes. However, it was pointed out by Slobodan Krstic (personal
communication) that the condition of that lemma is not correctly stated. Indeed,
the lemma fails if we consider $S$ consisting of two points and independent standard 
normal random variables indexed by the two points. We will make use of the following 
reformulation of the presentation of Talagrand's lower bound given by Ledoux \cite[(7.11)-(7.13)
on p. 257]{Ledoux}.


\begin{lemma} \label{Lem:Ta93}
Let $\{ X(t), t \in S \}$ be a separable, $\R^d$-valued, centered Gaussian process
indexed by a bounded set $S$ with the canonical metric $d_X(s, t) = (\E|X(s) - X(t)|^2)^{1/2}$.
Let $N_\varepsilon(S)$ denote the  smallest number of $d_X$-balls of radius $\varepsilon$
needed to cover $S$. If there is a decreasing function $\psi : (0, \delta] \to (0, \infty)$ such that 
$N_\varepsilon(S) \le \psi(\varepsilon)$ for all $\varepsilon \in (0, \delta]$ and there are constants 
$c_2 \ge c_1 > 1$ such that
\begin{equation}\label{Eq:Covering}
c_1 \psi (\varepsilon) \le \psi (\varepsilon/2) \le c_2 \psi (\varepsilon)
 \end{equation}
for all $\varepsilon \in (0, \delta]$, then there is a constant $K$ depending only
on $c_1$, $c_2$ and $d$ such that for all $u \in (0, \delta)$,
\begin{equation}\label{Eq:SB1}
\P\bigg( \sup_{s, t \in S} |X(s) - X(t)| \le u \bigg) \ge \exp \big(-K \psi(u) \big).
\end{equation}
\end{lemma}

Let $\rho \in (0, \rho_0/3)$ be a constant and let $(t^1, \dots, t^m) \in A_n$. Recall that $B^1_{2\rho}, \dots,
B^m_{2\rho}$ are the rectangles centered at $t^1, \ldots, t^m$.
By applying Assumption \ref{a1} and Lemma \ref{Lem:Ta93}, we derive the following lemma.
\begin{lemma}\label{lemma2}
Suppose that Assumption \ref{a1} holds and $\rho \in (0, \rho_0/3)$ is a constant.
Then there exist constants $K$ and $0 < \eta_0 < \rho_0/3$, depending on $c_0$ in
Assumption \ref{a1}, such that for all $(s^1, \dots, s^m) \in B^1_{2\rho} \times
\dots \times B^m_{2\rho}$, for all $0 < a < b$ and $0 < u < r < \eta_0$, we have
\begin{equation}\label{Eq:SB2}
\P\bigg(\sup_{1 \le i \le m} \, \sup_{x^i \in S(s^i, r)} | v([a, b), x^i) -
v([a, b), s^i) |  \le u \bigg) \ge \exp\left(-K \frac{r^{Q}}{u^{Q}} \right),
\end{equation}
where $Q  = \sum_{j = 1}^k \alpha_j^{-1}$.
\end{lemma}

\begin{proof} As suggested by the proof of (3.3) in Talagrand \cite{T98}, (\ref{Eq:SB2})
can be derived from Lemma \ref{Lem:Ta93}. However, there was a typo in the exponent in (3.3)
in \cite{T98} (the ratio $\frac{r} {u^{1/\alpha}}$ there should be raised to the power $N$) and the suggested proof
by introducing the auxiliary process $Z$ does not give the correct power for $\frac{r}
{u^{1/\alpha}}$  in (3.3) in \cite{T98}, which is needed for proving Proposition 3.4 in
\cite{T98}. Hence we give a proof of \eqref{Eq:SB2}.

For $(s^1, \dots, s^m) \in B^1_{2\rho}
\times \dots \times B^m_{2\rho}$ and $ r < \rho_0/3$, define $S = \bigcup_{i=1}^m {S}(s^i, r)$.
Under our assumption, we have ${S}(s^i, r) \subseteq T$ for $i = 1, \dots, m$. Thus, $S \subseteq T$.
It follows from Assumption \ref{a1} that for all $x, \, y \in S$,
\begin{align*}
\|v([a, b), x) - v([a, b), y)\|_{L^2}^2 &= \|v(x) - v(y)\|_{L^2}^2 -
\|v(\mathbb{R}_+\setminus [a, b), x) - v(\mathbb{R}_+ \setminus [a, b), y)\|_{L^2}^2\\
& \le \|v(x) - v(y)\|_{L^2}^2.
\end{align*}
By Lemma \ref{lemma1}, we have that the canonical metric for $\{v([a, b), x), x \in S\}$
satisfies
$$
d_{v}(s, t) := \|v([a, b), x) - v([a, b), y) \|_{L^2} \le 4 c_0  \Delta(x, y)
$$
for all $x, \, y \in S$ with  $ \Delta(x, y)$ small. Hence there is a constant
$\eta_0 \in (0, \rho_0/3)$ such that for all $r \in (0, \eta_0)$ and $\varepsilon \le r$,
the minimal number of $d_v$-balls of radius $\varepsilon$ needed to cover $S$ is
$$N_\varepsilon(S) \le \psi(\varepsilon) := C_{N, Q} \Big(\frac r \varepsilon \Big)^{Q}.$$
Note that this function $\psi(\varepsilon)$ satisfies \eqref{Eq:Covering} with the constants
$c_1 = c_2 = 2^{Q}$ which are greater than 1. It follows from Lemma \ref{Lem:Ta93} that there
is a constant $K$ such that \eqref{Eq:SB2} holds. This proves Lemma \ref{lemma2}.
\end{proof}

The following is the main estimate, which is an extension of Proposition 3.4 in
Talagrand \cite{T98}.

\begin{proposition}\label{prop}
Let $c > 0$ be a constant and suppose that Assumption \ref{a1} holds.
Then there are constants $K_1$ and $0 < \eta_1 < 1$
such that for all $0 < r_0 < \eta_1$, $\rho \in (0, \rho_0/3)$, $(t^1, \dots, t^m) \in \overline{A}_n$
and $(s^1, \dots, s^m) \in B^1_{2\rho} \times \dots \times B^m_{2\rho}$,
we have
\[
\begin{split}
\P\Bigg( \exists \, r \in [r_0^2, r_0], \, \sup_{1 \le i \le m} \,
\sup_{x^i \in {S}(s^i,\, c r)} |v(x^i) - v(s^i)| \le
K_1 r \bigg(\log\log\frac{1}{r} \bigg)^{-1/Q} \Bigg)\\
\ge 1 - \exp\Bigg(- \bigg(\log\frac{1}{r_0}\bigg)^{1/2}\Bigg).
\end{split}
\]
\end{proposition}

\begin{proof} The method of proof is similar to that of Proposition 3.4
in Talagrand \cite{T98}. 
For reader's convenience we provide a complete proof of Proposition \ref{prop} here.
The main ingredients are the small ball probability estimate in Lemma
\ref{lemma2} and the estimate of the approximation error in Lemma \ref{lem3.2}.

As in \cite{T95,T98} and \cite{DMX17}, let $U > 1$ be fixed for now and
its value will be chosen later. Set $r_\ell = r_0 U^{-2\ell}$ and
$a_\ell = U^{2\ell-1}/r_0$. Consider the largest integer $\ell_0$ such that
\begin{equation}\label{b6}
\ell_0 \leq \frac{\log(1/r_0)}{2 \log U}.
\end{equation}
Then for $\ell \leq \ell_0$, we have $r_\ell \geq r_0^2$.
It suffices to show that, for some large constant $K_1$,
\begin{align*}
&\P\left( \exists 1 \le \ell \le \ell_0, \, \sup_{1 \le i \le m}\,
\sup_{x^i \in S(s^i,\, c r_\ell)} |v(x^i) - v(s^i)|
\le K_1 \frac{r_\ell}{(\log\log\frac{1}{r_\ell})^{1/Q}} \right)\\
& \qquad \ge 1- \exp\left( -\left(\log\frac{1}{r_0}\right)^{1/2}\right).
\end{align*}

It follows from Lemma \ref{lemma2} that, for $K_1$ large enough so that
$K/K_1^Q \le 1/4$,
\begin{align}
\begin{aligned}\label{SmallBallLB}
&\P\left( \sup_{1 \le i \le m} \, \sup_{x^i \in S(s^i,\, c r_\ell)}
|v([a_\ell, a_{\ell+1}), x^i) - v([a_\ell, a_{\ell+1}), s^i)|
 \le K_1 \frac{r_\ell}{(\log\log\frac{1}{r_\ell})^{1/Q}} \right)\\
& \qquad \ge \exp\left( - \frac{K}{K_1^Q} \log\log\frac{1}{r_\ell} \right)\\
& \qquad \ge \left( \log\frac{1}{r_\ell} \right)^{-1/4}.
\end{aligned}
\end{align}
Thus, by the independence of the Gaussian processes $v([a_\ell, a_{\ell+1}), \cdot)$
($\ell = 1, \dots, \ell_0$), we have
\begin{align*}
& \P\left( \exists \ell \le \ell_0, \sup_{1 \le i \le m}\, \sup_{x^i \in S(s^i,\, c r_\ell)}
|v([a_\ell, a_{\ell+1}), x^i) - v([a_\ell, a_{\ell+1}), s^i)|
\le K_1 \frac{r_\ell}{(\log\log\frac{1}{r_\ell})^{1/Q}} \right)\\
& =1 - \prod_{\ell = 1}^{\ell_0} \left\{1 - \P\bigg( \sup_{1 \le i \le m} \,
\sup_{x^i \in S(s^i,\, c r_\ell)} |v([a_\ell, a_{\ell+1}), x^i) - v([a_\ell, a_{\ell+1}), s^i)|
\le K_1 \frac{r_\ell}{(\log\log\frac{1}{r_\ell})^{1/Q}} \bigg) \right\}.
\end{align*}
By \eqref{SmallBallLB}, we see that the last expression is greater than or equal to
\begin{align}
\begin{aligned}\label{LB1-exp}
1 - \prod_{\ell = 1}^{\ell_0} \left\{ 1 - \left( \log\frac{1}{r_\ell} \right)^{-1/4} \right\}
& \ge 1 - \left\{ 1 - \left( \log\frac{1}{r_0^2} \right)^{-1/4} \right\}^{\ell_0}\\
& \ge 1 - \exp\left( - \ell_0 \left( \log\frac{1}{r_0^2} \right)^{-1/4} \right).
\end{aligned}
\end{align}

Set
\[
A_\ell = \sum_{j=1}^k a_\ell^{\alpha_j^{-1} - 1} r_\ell^{\alpha_j^{-1}} + a_{\ell+1}^{-1}.
\]
Notice that $r_\ell a_\ell = U^{-1}$ and $r_\ell a_{\ell+1} = U$. Then
\begin{equation}\label{Eq:A2}
A_\ell r_\ell^{-1} = \sum_{j=1}^k (a_\ell r_\ell)^{\alpha_j^{-1}-1} + (a_{\ell+1} r_\ell)^{-1}
= \sum_{j=1}^k U^{-(\alpha_j^{-1}-1)} + U^{-1} \le (k+1) U^{-\beta},
\end{equation}
with $\beta = \min\{1, \min_{j=1, \dots, k} (\alpha_j^{-1}-1) \} > 0$
since $\alpha_j < 1$ for  $j=1, \dots, k$. Therefore, for $U$ large enough, $A_\ell
\le A_0r_\ell$, and for $u \ge \tilde{K} r_\ell U^{-\beta} \sqrt{\log U}$, \eqref{b1} is
satisfied. Hence, by Lemma \ref{lem3.2} and (\ref{Eq:A2}),
\begin{align*}
& \P\Bigg( \sup_{1 \le i \le m}\, \sup_{x^i \in S(s^i,\, c r_\ell)}
\big|v(x^i) - v(s^i) - v([a_\ell, a_{\ell+1}), x^i) + v([a_\ell, a_{\ell+1}), s^i)\big| \ge u \Bigg)\\
& \qquad \le \exp\bigg( - \frac{u^2}{\tilde{c} A_\ell^2} \bigg)\\
& \qquad \le \exp\bigg( -\frac{u^2}{\tilde{c}(k+1)^2 r_\ell^2} U^{2\beta} \bigg).
\end{align*}
Now we take $u = K_1 r_\ell (\log\log\frac{1}{r_0})^{-1/Q}$, which is allowed provided
\[
K_1 r_\ell \bigg( \log\log\frac{1}{r_0} \bigg)^{-1/Q} \ge \tilde{K} r_\ell U^{-\beta}\sqrt{\log U}.
\]
This is equivalent to
\begin{equation}\label{condU}
U^\beta(\log U)^{-1/2} \ge \frac{\tilde{K}}{K_1} \bigg( \log\log\frac{1}{r_0}\bigg)^{1/Q},
\end{equation}
which holds if  $U$ is large enough. It follows from the above that
\begin{align}
\begin{aligned}\label{UBexp}
&\P\Bigg( \sup_{1 \le i \le m} \sup_{x^i \in S(s^i, \, c r_\ell)}
\big|v(x^i) - v(s^i) - v([a_\ell, a_{\ell+1}), x^i)
+ v([a_\ell, a_{\ell+1}), s^i) \big| \ge  \frac{K_1 r_\ell}{( \log\log \frac{1}{r_0})^{1/Q}} \Bigg)\\
& \quad \le \exp \Bigg( - \frac{U^{2\beta}}{\tilde{c}(k+1)^2(\log\log\frac{1}{r_0})^{2/Q}} \Bigg).
\end{aligned}
\end{align}

Let
\begin{align*}
F_\ell & = \left\{ \sup_{1 \le i \le m} \sup_{x^i \in S(s^i,\, c r_\ell)} |v([a_\ell, a_{\ell+1}), x^i)
- v([a_\ell, a_{\ell+1}), s^i)| \le  \frac{K_1\,r_\ell}{2(\log\log\frac{1}{r_\ell})^{1/Q}}\right\},\\
G_\ell & = \left\{ \sup_{1 \le i \le m} \sup_{x^i \in S(s^i,\, c r_\ell)} |v(x^i) - v(s^i)
- v([a_\ell, a_{\ell+1}), x^i) + v([a_\ell, a_{\ell+1}), s^i)| \ge  \frac{K_1\, r_\ell}{2(\log\log\frac{1}{r_\ell})^{1/Q}}\right\}.
\end{align*}
Then
\begin{align}\label{LBP-P}
\begin{aligned}
& \P\left( \exists 1 \le \ell \le \ell_0, \sup_{1\le i \le m} \sup_{x^i \in S(s^i,\, c r_\ell)}
|v(x^i) - v(s^i)| \le \frac{K_1\, r_\ell}{(\log\log\frac{1}{r_\ell})^{1/Q}} \right)\\
& \qquad \ge \P\Bigg( \bigcup_{\ell=1}^{\ell_0} (F_\ell \cap G_\ell^c) \Bigg)\\
& \qquad \ge \P\left( \bigg( \bigcup_{\ell=1}^{\ell_0} F_\ell \bigg) \cap \bigg( \bigcup_{\ell=1}^{\ell_0} G_\ell\bigg)^c \right)\\
& \qquad \ge \P\left( \bigcup_{\ell=1}^{\ell_0}F_\ell\right) -
\P\left( \bigcup_{\ell = 1}^{\ell_0} G_\ell \right).
\end{aligned}
\end{align}
By \eqref{LB1-exp}, we have
\[
\P\left(\bigcup_{\ell=1}^{\ell_0}F_\ell \right) \ge 1 -
\exp\left( - \ell_0 \left( \log\frac{1}{r_0^2}\right)^{-1/4} \right),
\]
and by \eqref{UBexp},
\[
\P\left(\bigcup_{\ell=1}^{\ell_0} G_\ell \right) \le \ell_0 \exp\left( - \frac{U^{2\beta}} 
{\tilde{c}(k+1)^2(\log\log\frac{1}{r_0})^{2/Q}} \right).
\]
Combining this with \eqref{LBP-P}, we get
\begin{align*}
&\P\left( \exists 1 \le \ell \le \ell_0, \sup_{1 \le i \le m} \, \sup_{x^i \in S(s^i,\, c r_\ell)}
|v(x^i) - v(s^i)| \le  \frac{K_1\, r_\ell}{(\log\log\frac{1}{r_\ell})^{1/Q}}\right)\\
& \quad \ge1 - \exp\left( - \ell_0 \left( \log\frac{1}{r_0^2} \right)^{-1/4} \right) - \ell_0
\exp\left( - \frac{U^{2\beta}}{\tilde{c}(k+1)^2(\log\log\frac{1}{r_0})^{2/Q}} \right).
\end{align*}
Therefore, the proof will be completed provided
\begin{equation}\label{expUBcond}
\exp\left( - \ell_0 \left(\log\frac{1}{r_0^2}\right)^{-1/4} \right) + \ell_0 \exp\left( - \frac{U^{2\beta}}
{\tilde{c}(k+1)^2(\log\log\frac{1}{r_0})^{2/Q}} \right)
\le \exp\left( - \left(\log\frac{1}{r_0}\right)^{1/2} \right).
\end{equation}

Recall the condition \eqref{condU}, and the definition of $\ell_0$ in \eqref{b6}.
If we set
\[
U = \left( \log\frac{1}{r_0}\right)^{1/(2\beta)},
\]
then for $r_0$ small enough, by \eqref{b6},
\[
\ell_0 > \frac{\beta}{2} \left(\log\frac{1}{r_0}\right) \left( \log\log\frac{1}{r_0}\right)^{-1} > 1.
\]
Therefore, the left-hand side of \eqref{expUBcond} is bounded above by
\begin{align*}
& \exp\left( - \frac{(\log\frac{1}{r_0})^{3/4}}{\tilde{c}(k+1)^2\log\log\frac{1}{r_0}} \right)
+ \left( 1 + \log\frac{1}{r_0} \right)
\exp\left( - \frac{\log\frac{1}{r_0}}{\tilde{c}(k+1)^2(\log\log\frac{1}{r_0})^{2/Q}} \right)\\
& \qquad \le \exp\left( -\left(\log\frac{1}{r_0}\right)^{1/2} \right)
\end{align*}
provided $r_0$ is small enough. This completes the proof of Proposition \ref{prop}.
\end{proof}

Let $n \ge 1$ be fixed. Notice that if $(t^1, \dots, t^m) \in A_n$ and 
$0 < \rho < \frac{1}{2c'n}$, where $c'>0$ is the constant in Assumption \ref{a2}, then 
$\Delta(t^i, t^h) \ge 2 c'\rho$ for $i\ne h$. It follows from Assumption \ref{a2} that for 
each $(t^1, \dots, t^m) \in A_n$ and $\rho \in (0,\, \varepsilon_0 \wedge \frac{1}{2c'n})$, 
there are $(\hat{t}^1, \dots, \hat{t}^m) \in  B^1_{c\rho} 
\times \dots \times B^m_{c\rho}$ (recall that $B^i_{\rho}=B_{\rho}(t^i)$)
such that for all $h = 1, \dots, m$ and all $x, \,y \in B^i_{2\rho}$ ($i = 1, \dots, m$), 
we have
\begin{equation}\label{3.3}
\big|\E \big((v(x) - v(y)) \cdot v(\hat{t}^h)\big) \big| \le
C \sum_{j=1}^k |x_j - y_j|^{\delta_j}.
\end{equation}
The points $\hat{t}^1, \dots, \hat{t}^m$ are determined by $t^1, \dots, t^m$, and $(\hat{t}^1, 
\dots, \hat{t}^m) \in \overline{A}_{2n}$ provided $0 < \rho < \varepsilon_0 \wedge \frac 1 {2c'n}$,
where $a \wedge b = \min\{a, b\}$.

Let $\Sigma_2$ denote the $\sigma$-algebra generated by $v(\hat{t}^1),
\dots, v(\hat{t}^m)$. Define
\begin{equation}\label{Def:v2}
v^2(x) = \E \big(v(x)| \Sigma_2\big), \quad v^1(x) = v(x) - v^2(x).
\end{equation}
The Gaussian random fields $v^1= \{v^1(x), x \in T\}$ and $v^2 =
\{v^2(x), x \in T\}$ are independent.

\begin{lemma}\label{lemma4}
Suppose Assumptions \ref{a1}, \ref{a2} and \ref{a3} are satisfied.
For any $0< \rho < \varepsilon_0\wedge \frac 1 {2c'n}$, there is a constant $K_2$ 
depending on $n$ and the constants $C$
in Assumption \ref{a2} (but not on $\hat{t}^1, \dots, \hat{t}^m$) such that for all 
$i = 1, \dots, m$ and all $x, y \in B^i_{2\rho}$,
\[
\big|v^2(x) - v^2(y) \big| \le K_2 \sum_{j = 1}^k |x_j - y_j|^{\delta_j}
\max_{1 \le \ell \le m} \big|v(\hat{t}^\ell) \big|.
\]
\end{lemma}

\begin{proof}
By Assumption \ref{a3}, the subspace in $L^2(\Omega; \mathbb{R}^d)$ of random vectors 
$\Omega \to \mathbb R^d$ spanned by $v(\hat{t}^1), \dots, v(\hat{t}^m)$, has dimension $m\geq 2$.
Let $\big\{ \sum_{h = 1}^m a_{h,j}v(\hat{t}^h) : j = 1, \dots, m \big\}$ be an orthonormal basis of
this subspace obtained by the Gram--Schmidt orthogonalization procedure, so that the coefficients 
$a_{i,j}$ are continuous functions of 
$(\hat{t}^1, \dots, \hat{t}^m) \in \overline{A}_{2n}$.
Then
\begin{align*}
v^2(x) &= \sum_{j=1}^m \E\bigg[\sum_{h = 1}^m a_{h,j}v(\hat{t}^h) \cdot v(x)\bigg]
\bigg(\sum_{\ell = 1}^m a_{\ell, j}v(\hat{t}^\ell)\bigg).
\end{align*}
By \eqref{3.3}, for all $i = 1, \dots, m$ and all $x, y \in B^i_{2\rho}$,
\begin{align*}
\big|v^2(x) - v^2(y) \big| &= \Bigg|\sum_{\ell = 1}^m \bigg( \sum_{h=1}^m \sum_{j=1}^m
a_{h, j}a_{\ell, j}  \E\left[ (v(x) - v(y)) \cdot v(\hat{t}^h) \right] \bigg) v(\hat{t}^\ell)\Bigg|\\
& \le K \sum_{j=1}^k |x_j - y_j|^{\delta_j} \max_{1 \le \ell \le m} \big| v(\hat{t}^\ell) \big|.
\end{align*}
By the continuity of  $a_{i,j}$ in $(\hat{t}^1, \ldots, \hat{t}^m)$ and the compactness of 
$\overline{A}_{2n}$, we see that  the constant $K$  is independent of $\hat{t}^1, \dots, \hat{t}^m$. 
This completes the proof.
\end{proof}

\begin{lemma}\label{lemma5}
Suppose 
Assumptions \ref{a1}, \ref{a2} and \ref{a3} are satisfied.
Let $n \ge 1$.
Then there exist constants $K$ and $\rho_0 \in (0, 1/n)$ which may depend on $n$ such
that for all $\rho \in (0, \rho_0)$,  $a_2, \dots, a_m \in \mathbb{R}^d$,
$r > 0$,  $(t^1, \dots, t^m ) \in \overline{A}_n$ and all $(x^1, \dots, x^m) 
\in B_\rho(t^1) \times \dots \times B_\rho(t^m)$,
\[
\P\left( \sup_{2 \le i \le m} |v^2(x^1) - v^2(x^i) - a_i| \le r \right) \le K r^{(m-1)d}.
\]
\end{lemma}

\begin{proof}
We first assume $d = 1$.
We claim that if $\rho_0$ is small then $v^2(x^1), \dots, v^2(x^m)$ are linearly independent for all $\rho \in
(0, \rho_0)$, $(t^1, \dots, t^m) \in \overline{A}_n$ and $(x^1, \dots, x^m) \in B_\rho(t^1) \times \dots \times 
B_\rho(t^m)$. Indeed, by Assumption \ref{a3} and the compactness of  $\overline{A}_n$, we
can find $C > 0$ depending on $n$ such that $\mathrm{Var}(\sum_{i = 1}^m b_i v({t}^i)) \ge C$ for all 
$(t^1, \dots, t^m ) \in \overline{A}_n$ and $b \in \mathbb{R}^m$ with $|b| = 1$.
Then by the Cauchy--Schwarz inequality, we have
\begin{align*}
&\Bigg[\E\bigg( \sum_{i=1}^m b_i (v(t^i) - v^2(x^i)) \bigg)^2\Bigg]^{1/2}
 \le |b| \Bigg[\E\left(\sum_{i=1}^m \left(v(t^i) - v^2(x^i)\right)^2 \right)\Bigg]^{1/2}\\
&\qquad \le |b| \sum_{i=1}^m \left( \left[ \E\left(v(t^i) - v(\hat{t}^i) \right)^2\right]^{1/2} + \left[ \E\left( \E(v(\hat{t}^i) - v(x^i)|
\Sigma_2) \right)^2 \right]^{1/2} \right)\\
& \qquad \le |b| \sum_{i=1}^m\Big( \|v(t^i) - v(\hat{t}^i) \|_{L^2} + \|v(\hat{t}^i) - v(x^i) \|_{L^2} \Big).
\end{align*}
It follows that
\begin{align*}
\Bigg[\E \bigg( \sum_{i=1}^m b_i v^2(x^i) \bigg)\Bigg]^{1/2}
&\ge \Bigg[ \E\bigg( \sum_{i=1}^m b_i v(t^i) \bigg)^2 \Bigg]^{1/2}
- \Bigg[ \E \bigg( \sum_{i=1}^m b_i (v(t^i) - v^2(x^i)) \bigg)^2 \Bigg]^{1/2}\\
& \ge \Bigg( C^{1/2} - \sum_{i=1}^m\Big( \|v(t^i) - v(\hat{t}^i) \|_{L^2} +
\|v(\hat{t}^i) - v(x^i) \|_{L^2} \Big) \Bigg) |b|.
\end{align*}
Notice that, Assumption \ref{a1} implies  the $L^2(\P)$-continuity of $v(x)$
[cf. Lemma \ref{lemma1}], we can find a
small constant $\rho_0 \in (0, 1/n)$ depending on $C$ so that the above is $\ge C'|b|$ 
for all $\rho \in (0, \rho_0)$, $(t^1, \dots, t^m) \in \overline{A}_n$ and $(x^1, \dots, x^m) 
\in B_\rho(t^1) \times \dots \times B_\rho(t^m)$, where $C' > 0$. It follows that $v^2(x^1), 
\dots, v^2(x^m)$ are linearly independent, and so are $v^2(x^1) - v^2(x^2),
v^2(x^1) - v^2(x^3), \dots, v^2(x^1) - v^2(x^m)$.

Denote the determinant of the covariance matrix of the last random vector by
$$
\det \mathrm{Cov} (v^2(x^1) - v^2(x^2), v^2(x^1) - v^2(x^3), \dots, v^2(x^1) - v^2(x^m)).
$$
If $\rho \in (0, \rho_0)$, $(t^1, \dots, t^m) \in \overline{A}_n$ and $(x^1, \dots, x^m)
 \in B_\rho(t^1) \times \dots \times B_\rho(t^m)$, then $(x^1, \dots, x^m) \in \overline{A}_{2n}$ 
 provided that $\rho_0$ is small. Since the function $(x^1, \dots, x^m) \mapsto \det \mathrm{Cov}
(v^2(x^1) - v^2(x^2), v^2(x^1) - v^2(x^3), \dots, v^2(x^1) - v^2(x^m))$ is
continuous and positive on the compact set $\overline{A}_{2n}$,
it is bounded from below by a positive constant depending on $n$.
This and Anderson's theorem \cite{Anderson} imply that
\[
\P\left( \sup_{2 \le i \le m} |v^2(x^1) - v^2(x^i) - a_i| \le r \right) \le
\P\left( \sup_{2 \le i \le m} |v^2(x^1) - v^2(x^i)| \le r \right)
\le K r^{m-1}.
\]
Since $v(x)$ has i.i.d.\ components, the case $d > 1$ follows readily.
\end{proof}

We end this section with the following lemma which is obtained by applying
Theorem 2.1 and Remark 2.2 of \cite{KRS12} to the metric space $(T, \Delta)$.
It provides nested families of ``cubes" sharing most of the good properties of
dyadic cubes in the Euclidean spaces. For this reason, we call the sets in
$\mathscr{Q}_q$ generalized dyadic cubes of order $q$. Their nesting property
will help us to construct an economic covering for $M_\rho$.

\begin{lemma}\label{lemma6}
There exist constants $c_1, c_2$, and a family $\mathscr{Q}$ of Borel
subsets of $T$, where $\mathscr{Q}
= \bigcup_{q=1}^\infty \mathscr{Q}_q$, $\mathscr{Q}_q = \{ I_{q, \ell} :
\ell =1, \dots, n_q \}$, such that the following hold.
\begin{enumerate}
\item[$(i)$] $T = \bigcup_{\ell=1}^{n_q} I_{q, \ell}$ for each $q \ge 1$.
\item[$(ii)$] Either $I_{q, \ell} \cap I_{q', \ell'} = \varnothing$ or $I_{q, \ell}
\subset I_{q', \ell'}$ whenever $q \ge q'$, $1 \le \ell \le n_q$, $1 \le \ell' \le n_{q'}$.
\item[$(iii)$] For each $q, \ell$, there exists $x_{q, \ell} \in T$ such that
$ S(x_{q,\ell}, c_1 2^{-q}) \subset I_{q,\ell} \subset S(x_{q,\ell}, c_2 2^{-q})$
and $\{ x_{q,\ell} : 1, \dots, n_q \} \subset \{x_{q+1, \ell} : \ell = 1, \dots, n_{q+1} \}$ for
all $q \ge 1$.
\end{enumerate}
\end{lemma}

\section{Proof of Theorem \ref{main thm}}
\label{Sec:4}

Recall that, by \eqref{3.2}, it suffices to show that for all integers $n$, 
we can find a small $\rho_0 > 0$  such that for all $\rho \in (0, \rho_0)$
and all points $(t^1, \dots, t^m) \in A_n$, 
$M_\rho$ is empty with
probability 1. When $m Q < (m-1)d$ (we refer to this as the
sub-critical case), the last statement can be proved easily
by using a standard covering argument based on the uniform
modulus of continuity of $v = \{v(x), x \in T\}$ on compact intervals.
In the following we provide a unified proof for both the critical and
sub-critical cases.

For any $n \ge 1$ fixed, we choose a constant $\rho_0>0$ such that 
Assumption \ref{a2}, Lemma \ref{lemma4} and Lemma \ref{lemma5}
hold for all $\rho \le \rho_0$ (e.g., we take
$\rho_0 \le \varepsilon_0 \wedge \frac{1}{2c'n}$). Let $(t^1, \dots, t^m) \in A_n $ be 
fixed in the rest of the proof. 
By Assumption \ref{a2}, we can find $(\hat{t}^1, \dots, \hat{t}^m) \in
B^1_{c\rho} \times \dots \times   B^m_{c\rho}$ such that \eqref{3.3}
holds. Furthermore, we assume that 
$B^j_{c\rho_0}  \subset T$ for all $1 \le j \le m$ (otherwise we
take the intersection with $T$).

Fix $\rho \in (0, \rho_0)$. For each integer $p \ge 1$, consider the random set
\begin{equation} \label{Eq:Osc1}
\begin{split}
R_p = \Bigg\{ & (s^1, \dots, s^m) \in B^1_{2\rho} \times \dots \times B^m_{2\rho} :
\exists \, r \in [2^{-2p}, 2^{-p}] \text{ such that}\\
&\qquad  \sup_{1 \le i \le m} \, \sup_{x^i \in {S}(s^i, 4c_2r)} |v(x^i) - v(s^i)|
\le K_1 r \left( \log\log\frac{1}{r} \right)^{-1/Q} \Bigg\},
\end{split}
\end{equation}
where $c_2$ is the constant given by Lemma \ref{lemma6}.
Fix $\beta \in (0, \min\{\beta^*, 1\})$, where $\beta^* = \min\{ \delta_j/\alpha_j - 1 : j = 1, \dots, k\}$.
Let $\lambda$ denote the Lebesgue measure on $\mathbb{R}^{mk}$.
Consider  the event
\begin{align*}
& \Omega_{p, 1} = \Big\{ \lambda(R_p) \ge \lambda(B^1_{2\rho} \times \dots
\times B^m_{2\rho})(1 - \exp(-\sqrt{p}/4)) \Big\}.
\end{align*}
This event states that a very large portion of $B^1_{2\rho} \times \dots \times B^m_{2\rho}$
is taken by the random set $R_p$, which is the collection of points at 
which the sample function $v(x)$ has small oscillations. 
The points in $R_p$ are referred to as  ``good points" for $v$. By Markov's inequality, 
\begin{align*}
\P(\Omega_{p, 1}^c) &= \P\Big\{ \lambda(B^1_{2\rho} \times \dots \times B^m_{2\rho} \setminus R_p) > 
\lambda(B^1_{2\rho} \times \dots \times B^m_{2\rho}) \exp(-\sqrt{p}/4) \Big\}\\
& \le \frac{\E \big[\lambda(B^1_{2\rho} \times \dots \times B^m_{2\rho} \setminus R_p)\big]} 
{\lambda(B^1_{2\rho} \times \dots \times B^m_{2\rho}) \exp(-\sqrt{p}/4)}.
\end{align*}
Then by Fubini's theorem, the numerator is equal to
$$
\E \big[\lambda(B^1_{2\rho} \times \dots \times B^m_{2\rho} \setminus R_p)\big]
= \int_{B^1_{2\rho} \times \dots \times B^m_{2\rho}} \P\Big( (s^1, \dots, s^m) \in B^1_{2\rho} 
\times \dots \times B^m_{2\rho} \setminus R_p \Big)\, ds^1 \cdots ds^m.
$$
By applying Proposition \ref{prop} with $c = 4c_2$, we derive that for
$p$ sufficiently large,
$$
\P \Big((s^1, \dots, s^m) \in R_p \Big) \ge 1 - \exp(-\sqrt{p}/2)
$$
for all $(s^1, \dots, s^m)  \in B^1_{2\rho} \times \dots \times B^m_{2\rho}.$
It follows that $\P(\Omega_{p, 1}^c) \le \exp(-\sqrt{p}/4)$
and 
$\sum_{p = 1}^\infty \P(\Omega_{p,1}^c) < \infty$. 

Consider the event
$$
\Omega_{p, 2} = \left\{ \max_{1 \le i \le m} |v(\hat{t}^i)| \le 2^{\beta p} \right\}.
$$
By Lemma \ref{lemma4}, we can control the oscillation of $v^2(x)$ on $\Omega_{p, 2}$. It is clear that 
$\sum_{p=1}^\infty \P(\Omega_{p,2}^c) < \infty$.

For the points which are not in the random set $R_p$,  the sample function $v (x)$ may have large 
oscillations in their neighborhoods. These points are referred to as ``bad points" for $v$. In order to 
quantify their effect we introduce the following event $\Omega_{p, 3}$:
\begin{equation}\label{Eq:Osc2}
\Omega_{p, 3} = \bigg\{ \forall \, I \in \mathscr{Q}_{2p}, \,
\sup_{x, y \in I} |v(x) - v(y)| \le K_3 2^{-2p} p^{1/2} \bigg\}.
\end{equation}
Recall that $\mathscr{Q} = \bigcup_{q = 1}^\infty\mathscr{Q}_q$ is the family of
generalized dyadic cubes given by Lemma \ref{lemma6} for  the compact interval $T$.

For every $I  \in \mathscr{Q}_{2p}$, Lemma \ref{lemma1} implies that the
diameter of $I$ under the canonical metric $d_v(x, y) = \|v(x) - v(y)\|_{L^2}$ is at most
$c_3\, 2^{-2p}$. By applying Lemma 2.1 in Talagrand \cite{T95} (see also Lemma 3.1 in
\cite{DMX17}) we see that for any positive constant $K_3$ and $p$ large,
\[
\P\bigg(\sup_{x, y \in I} |v(x) - v(y)| \ge K_3 2^{-2p} p^{1/2} \bigg) \le
\exp \bigg(-\Big( \frac{K_3 }{c_3}\Big)^2 p\bigg).
\]
Notice that the cardinality of the family $\mathscr{Q}_{2p}$ of generalized dyadic cubes
of order $2p$ is at most $K 2^{2pQ}$. We can verify directly that
$\sum_{p=1}^\infty  \P(\Omega_{p,3}^c) < \infty$ provided  $K_3$ is chosen to satisfy
$K_3 > 2c_3 Q\ln 2 $.

Let $\Omega_p = \Omega_{p, 1} \cap \Omega_{p,2} \cap \Omega_{p,3}$ and
\[
 \Omega^* = \bigcup_{\ell \ge 1} \bigcap_{p \ge \ell} \Omega_p.
 \]
It follows from the above that the event $\Omega^*$ occurs with probability 1. Hence, 
almost surely $ \Omega_p$ occurs for all $p$ large enough. Notice that 
on the event $\Omega_p$, the oscillations of $v$ near the good and bad points can be
explicitly controlled by the inequalities in (\ref{Eq:Osc1}) and  (\ref{Eq:Osc2}), respectively. 
Moreover, because of Lemma  \ref{lemma4} and $\Omega_{p,2}$, it 
can be verified that similar inequalities (with constants larger than $K_1$ and $K_3$) hold 
for $v^1$.  In the following, we will use these observations to show that, for every 
$\omega \in \Omega^*$, we can construct families of balls in $\mathbb{R}^d$ that cover 
$M_\rho$.

For each $p \ge 1$, we first construct a family $\mathscr{G}_p$ of subsets
in $\mathbb{R}^{mk}$ (depending on $\omega$). Denote by $\mathscr{C}_p$
the family of subsets of $T^m$ of the form $C = I_{q, \ell_1}
\times \dots \times I_{q, \ell_m}$ for some integer $q \in [p, 2p]$,
where $I_{q, \ell_i} \in \mathscr{Q}_q$
are the generalized dyadic cubes of order $q$ in Lemma \ref{lemma6}. 

We say that a dyadic cube $C = I^1 \times \dots \times I^m$ of order
$q$ is \emph{good} if it has the property that
\begin{equation}\label{Eq:good}
\sup_{1 \le i \le m}\, \sup_{x, y \in I^i} |v^1(x) - v^1(y)| \le d_q,
\end{equation}
where
\begin{equation}\label{Eq:dq}
d_q = 2\bigg(K_1 + K_2 \sum_{j=1}^k(2c_2)^{\delta_j/\alpha_j} \bigg) 2^{-q}(\log\log 2^{q})^{-1/Q}.
\end{equation}
For each $x \in B^1_{2\rho} \times \dots \times B^m_{2\rho}$, consider the
good dyadic cube $C$ containing $x$ (if any) of smallest order $q$, where
$p \le q \le 2p$.  By property (ii) of Lemma \ref{lemma6}, we obtain in this way
a family of disjoint good dyadic cubes of order $q \in [p, 2p]$ that meet
the set $B^1_{2\rho} \times \dots \times B^m_{2\rho}$. We denote this
family by $\mathscr{G}_p^1$.

Let $\mathscr{G}_p^2$ be the family of dyadic cubes in $T^m$ of order $2p$
that meet $B^1_\rho \times \dots \times B^m_\rho$ but are not contained in
any cube of $\mathscr{G}_p^1$. 
Let $\mathscr{G}_p = \mathscr{G}_p^1 \cup
\mathscr{G}_p^2$. Notice that for each $C \in \mathscr{C}_p$, the events
$\{ C \in \mathscr{G}_p^1\}$ and $\{ C \in \mathscr{G}_p^2 \}$ are in the
$\sigma$-algebra $\Sigma_1 := \sigma\{v^1(x):\, x \in T\}$.

Next we construct a family of balls in $\mathbb{R}^d$
(depending on $\omega$) as follows. For each $C \in \mathscr{C}_p$, we
choose a distinguished (non-random) point $x_C = (x^1_C, \dots, x^m_C)$ in
$C \cap (B^1_{2\rho} \times \dots \times B^m_{2\rho})$. If $C$ is
a cube of order $q$, then we define the ball $B_{p, C}$  as follows.
\begin{enumerate}
\item[(i)] If $C \in \mathscr{G}_p^1$, take $B_{p, C}$ as the Euclidean ball in $\R^d$ 
of center $v(x^1_C)$ of radius $r_{p, C} = 4d_q$. Recall that $d_q$ is defined in 
(\ref{Eq:dq}).

\item[(ii)] If $C \in \mathscr{G}_p^2$, take $B_{p, C}$ as the Euclidean ball in $\R^d$
of center $v(x^1_C)$ of radius $r_{p, C} = 2K_3 2^{-2p}\,p^{1/2}$.
\item[(iii)] Otherwise, take $B_{p, C} = \varnothing$ and $r_{p, C} = 0$.
\end{enumerate}
Note that for each $p \ge1$ and $ C \in \mathscr{C}_p$, the random variable
$r_{p, C}$ is $\Sigma_1$-measurable. 

Intuitively, if a ball $B_{p, C}$ in $\R^d$ defined above contains a multiple point of $v$ from 
$C \cap (B^1_{2\rho} \times \dots \times B^m_{2\rho})$ (i.e., there exists  $(y^1, \dots, y^m) 
\in C \cap (B^1_{2\rho} \times \dots \times B^m_{2\rho})$ such that $v(y^1) = \cdots = v(y^m)
 \in B_{p, C}$), then  $v(x^1_C), v(x^2_C), \dots, v(x^m_C)$  should be close to each other 
 because $|v(x^i) - v(y^i)|$  is small for $1 \le i \le k$. Hence,  in order to construct  
a covering $\mathscr{F}_p(\omega)$ of the set $M_\rho$ of multiple points, we consider 
the event
\[ \Omega_{p, C} = \left\{ \omega \in \Omega : \sup_{2 \le i \le m}
\big| v(x^1_C, \omega) - v(x^i_C, \omega) \big| \le r_{p, C}(\omega) \right\}.
\]
This is the event that $v(x^2_C), \dots, v(x^m_C)$ are all within the ball of radius $r_{p, C}$ 
centered at $v(x^1_C)$.

If $\omega \in \Omega_p \cap \Omega_{p, C}$, define $\mathscr{F}_p(\omega) =
\{ B_{p, C} : C \in \mathscr{G}_p(\omega) \}$. 
Otherwise,  define $\mathscr{F}_p(\omega) = \varnothing$.

Choose an integer $p_0$ such that
\begin{equation}\label{p_0}
2c_2 2^{-p} \le \rho \ \hbox{ and }\ \exp(-\sqrt{p}/4) p^{mQ} (\log p)^{m} \le \rho^{mQ},
\end{equation}
and $\Omega_p$ occurs for all $p \ge p_0$.
We now show that $\mathscr{F}_p(\omega)$ covers $M_\rho(\omega)$ whenever $p \ge p_0$ 
and $\omega \in \Omega_p$.

Let $\omega \in \Omega_p$ and $z \in M_\rho(\omega)$. By definition, we can find a point $(y^1, \dots, y^m)
\in B^1_\rho \times \dots \times B^m_\rho$ such that $z = v(y^1, \omega) = \dots = v(y^m, \omega)$. By the
definitions of $\mathscr{G}_p^1$ and $\mathscr{G}_p^2$, the family $\mathscr{G}_p(\omega)$ of dyadic
cubes covers $B^1_\rho \times \dots \times B^m_\rho$, thus the point $(y^1, \dots, y^m)$ is contained
in some $C = I^1 \times \dots \times I^m \in \mathscr{G}_p(\omega)$. We will show that $z \in B_{p, C}$
and $\omega \in \Omega_{p, C}$. To this end, we distinguish two cases.

{\it Case 1.}\, If $C \in \mathscr{G}_p^1(\omega)$, then it is a good dyadic cube of order $q \in [p, \,2p]$
such that
$$\sup_{1 \le i \le m}|v^1(x^i_C, \omega) - v^1(y^i, \omega)| \le d_q.$$
By Lemma \ref{lemma6}, $x^i_C, y^i \in I^i \subset {S}(x^*, c_2 2^{-q})$ for some $x^* \in T$, so we have
\begin{equation}\label{Eq:xy}
\sum_{j=1}^k|x^i_{C, j} - y^i_j|^{\delta_j} \le \sum_{j=1}^k(2c_2)^{\delta_j/\alpha_j} 2^{-q(1 + \beta^*)},
\end{equation}
recall that $ \beta^* = \min\limits_{1 \le j \le k}\big\{\frac{\delta_j}{\alpha_j} - 1 \big\}$.
Since $\omega \in \Omega_{p,2}$, Lemma \ref{lemma4} and \eqref{Eq:xy} imply that
\begin{equation}\label{Eq:1star}
\sup_{1 \le i \le m} \big|v^2(x^i_C) - v^2(y^i)\big|
\le K_2\sum_{j=1}^k(2c_2)^{\delta_j/\alpha_j}2^{-q(1 + \beta^* - \beta)} \le d_q
\end{equation}
since $\beta < \beta^*$. It follows that
$$ \sup_{1 \le i\le m} \big|v(x^i_C, \omega) - z\big| =
\sup_{1 \le i \le m} \big |v(x^i_C, \omega) - v(y^i, \omega) \big| \le 2 d_q,$$
which implies that $z \in B_{p, C}$ and $\omega \in \Omega_{p, C}$.

{\it Case 2.}\, Now we assume $C \in \mathscr{G}_p^2(\omega)$. Since
$\omega \in \Omega_{p, 3}$, we have
$$\sup_i |v(x^i_C, \omega) - z| = \sup_{i} |v(x^i_C, \omega) - v(y^i, \omega)| \le K_3 2^{-2p} p^{1/2},$$
hence $z \in B_{p, C}$ and $\omega \in \Omega_{p, C}$.

Therefore, for every $\omega \in \Omega^*$, $\mathscr{F}_p(\omega)$ covers $M_\rho(\omega)$ for all $p$
large enough. We claim that, with probability 1, the family $\mathscr{F}_{p}$ is empty for infinitely
many $p$. This will imply that $M_\rho$ is empty with probability 1 and the proof will then be complete.

We prove the aforementioned claim by contradiction. Suppose the claim is not true. Then the event
$\Omega'$ that $\mathscr{F}_{p}$ is nonempty for all large $p$ has positive probability and the event
$\Omega' \cap \Omega^* = \bigcup_{\ell \ge 1} \bigcap_{p \ge \ell} (\Omega' \cap \Omega_{p})$ also has
positive probability. Denote
\[
\phi(r) = r^{mQ - (m-1)d}(\log\log(1/r))^{m}, \qquad f(r) = r^{mQ}(\log\log(1/r))^m,
\]
and consider the random variables $X_p$ defined by
\begin{equation}\label{Eq:Xp}
X_p := \1_{\Omega' \cap \Omega_p} \sum_{B_{p, C} \in \mathscr{F}_p} \phi(r_{p, C}) =
\1_{\Omega' \cap \Omega_p} \sum_{C \in \mathscr{C}_p} f(r_{p, C})r_{p, C}^{-(m-1)d}\1_{\{ C \in \mathscr{G}_p\}}
\1_{\Omega_{p, C}}.
\end{equation}

Let  $X := \liminf_{p} X_p$. Since $mQ \le (m-1)d$, we have $\phi(r) \to \infty$ as $r \to 0+$.
Also, for every $\omega \in \Omega' \cap \Omega^*$, 
$\mathscr{F}_{p}(\omega) $ is not empty for all large $p$. 
Hence, by the definition of $X_p$ in (\ref{Eq:Xp}), $X(\omega) = \infty$ on
$\Omega' \cap \Omega^*$. In particular, $\E(X) = \infty$.

On the other hand, notice that $\mathscr{G}_p^1$ covers $R_p$ on the event $\Omega_p$
for all $p \ge p_0$. Indeed, if $\omega \in \Omega_p$, $s = (s^1, \dots, s^m) \in R_p(\omega)$,
and $C = I^1 \times \dots \times I^m$ is the dyadic cube of order $q$ in $\mathscr{G}_p^1$ containing $s$,
then there exists $r \in [2^{-2p}, 2^{-p}]$ that satisfies the condition
in the definition of $R_p$ and we can find $q$ such that $2^{-q-1} < r \le 2^{-q}$,
$p \le q \le 2p$, and
\begin{equation}\label{Eq:2star}
\sup_{1 \le i \le m} \, \sup_{x^i \in {S}(s^i, 2c_2 2^{-q})} |v(x^i) - v(s^i)|
\le K_1 2^{-q} (\log\log2^q)^{-1/Q}.
\end{equation}
By the property that $I^i \subset {S}(x', c_2 2^{-q})$ for some $x'$ and by Lemma
\ref{lemma4}, it follows from (\ref{Eq:1star}) and  (\ref{Eq:2star}) that (\ref{Eq:good}) holds.
Thus $C$ is a good dyadic cube. This proves that $\mathscr{G}_p^1(\omega)$ covers
$R_p(\omega)$.

By the choice of $p_0$ in \eqref{p_0}, the cubes in $\mathscr{G}_p^2$ are contained in $B^1_{2\rho}
\times \dots \times B^m_{2\rho}$, and thus in $B^1_{2\rho} \times \dots \times B^m_{2\rho}
\setminus R_p$, the Lebesgue measure of which is at most $\exp(-\sqrt{p}/4)$ on $\Omega_p$.
For any $C = I^1 \times \dots \times I^m \in \mathscr{G}_p^2$ of order $2p$, each $I^i$
contains a set $S(x^i, c_1 2^{-2p})$ for some $x^i$ and the set has Lebesgue measure
$K2^{-2pQ}$, so $\Omega_p$ is contained in the event $\widetilde{\Omega}_p$ that the
cardinality of $\mathscr{G}_p^2$ is at most $K 2^{2pmQ}\exp(-\sqrt{p}/4)$.

Recall that both  $\mathscr{G}_p^1 $ and $ \mathscr{G}_p^2 $ depend on $\Sigma_1$.
We see that  $\widetilde{\Omega}_p$ belongs to the $\sigma$-algebra $\Sigma_1$. Hence
for $p \ge p_0$,
\begin{equation} \label{Eq:Xp2}
\begin{split}
\E(X_p) &\le \E\bigg( \1_{\widetilde{\Omega}_p} \sum_{C \in \mathscr{C}_p} f(r_{p, C}) r_{p, C}^{-(m-1)d}
\1_{\{C \in \mathscr{G}_p\}} \1_{\Omega_{p, C}} \bigg)\\
& = \E\bigg( \1_{\widetilde{\Omega}_p} \sum_{C \in \mathscr{C}_p} f(r_{p, C}) r_{p, C}^{-(m-1)d}
\1_{\{ C \in \mathscr{G}_p\}} \P(\Omega_{p, C} | \Sigma_1) \bigg)\\
& \le K \E\bigg( \1_{\widetilde{\Omega}_p} \sum_{C \in \mathscr{C}_p} f(r_{p, C})
\1_{\{ C \in \mathscr{G}_p\}} \bigg),
\end{split}
\end{equation}
where the last inequality follows from Lemma \ref{lemma5} and independence of $v^1$ and $v^2$.

Now consider any dyadic cube $C \in \mathscr{C}_p$ of order $q$. If $C \in \mathscr{G}_p^1$, then
$f(r_{p, C}) \le K 2^{-qmQ} \le K \lambda(C)$ (where $\lambda(\cdot)$ denotes Lebesgue measure); 
if $C \in \mathscr{G}_p^2$, then $f(r_{p, C})
\le K 2^{-2pmQ} p^{mQ/2}(\log p)^{m}$. Moreover, for $p \ge p_0$ the dyadic cubes in $\mathscr{G}_p^1$
are disjoint and contained in $B^1_{3\rho} \times \dots \times B^m_{3\rho}$. These observations,
together with (\ref{Eq:Xp2}) and \eqref{p_0}, imply
\[
\E(X_p) \le K\, \E\bigg( \sum_{C \in \mathscr{C}_p} \lambda(C) \1_{\{ C \in \mathscr{G}_p^1\}}
 + p^{mQ/2} (\log p)^m\,  \exp(-\sqrt{p}/4) \bigg) \le K \rho^{mQ}.
\]
By Fatou's lemma, we derive $\E(X) \le K \rho^{mQ}< \infty$. This is a contradiction.
The proof of Theorem \ref{main thm} is complete.
\qed

\section{Examples}
In this section we provide some examples where Theorem \ref{main thm} is applicable. These include fractional
Brownian sheets, and the solutions to systems of stochastic heat and wave equations.

\subsection{Fractional Brownian sheets}

The $(N, d)$-fractional Brownian sheet with Hurst parameter $H = (H_1, \dots, H_N)\in (0, 1)^N$
is an $\mathbb{R}^d$-valued continuous Gaussian random field $\{v(x), x \in \mathbb{R}^N_+ \}$ with
mean zero and covariance
\begin{equation*}
\E(v_j(x)v_\ell(y)) = \delta_{j,\ell} \prod_{i = 1}^N \frac{1}{2}\left(|x_i|^{2H_i} + |y_i|^{2H_i} - |x_i - y_i|^{2H_i}\right).
\end{equation*}
When $N = 1$, it is the fractional Brownian motion and the non-existence of
multiple points in the critical dimension was proved by Talagrand \cite{T98}.
So we focus on the case $N \ge 2$.

Let $\alpha \in (0, 1)$ be a constant. We start with the identity that any $x \in \mathbb{R}$,
\begin{equation*}
|x|^{2\alpha} = c^2_\alpha \int_{\mathbb{R}} \frac{1-\cos x\xi}{|\xi|^{2\alpha+1}} \,d\xi, \quad
\hbox{ where }\ c_\alpha = \left( \int_{\mathbb{R}} \frac{1-\cos \xi}{|\xi|^{2\alpha+1}} d\xi \right)^{-1/2},
\end{equation*}
which can be obtained by a change of variable in the integral.
It implies that for any $x, y \in \mathbb{R}$,
\begin{equation*}
\frac{1}{2}\left( |x|^{2\alpha} + |y|^{2\alpha} - |x - y|^{2\alpha} \right) = c_\alpha^2
\int_{\mathbb{R}}\left[\frac{(1 - \cos{x \xi})(1 - \cos{y \xi})}
{|\xi|^{2\alpha+1}} + \frac{\sin{x\xi}\sin{y\xi}}{|\xi|^{2\alpha+1}} \right]\,d\xi.
\end{equation*}
It follows that for $H \in (0, 1)^N$ and $x, y \in \mathbb{R}^N_+$,
we can write
\begin{equation}\label{5.1}
\prod_{i=1}^N\frac{1}{2}\left( |x_i|^{2H_i} + |y_i|^{2H_i} - |x_i - y_i|^{2H_i} \right)
= c_H^2 \sum_{p \in \{0, 1\}^N}\int_{\mathbb{R}^N}
\prod_{i=1}^N \frac{f_{p_i}(x_i \xi_i)f_{p_i}(y_i \xi_i)}{|\xi_i|^{2H_i + 1}} \, d\xi,
\end{equation}
where $f_0(t) = 1 - \cos{t}$ and $f_1(t) = \sin{t}$. It gives a representation for the
fractional Brownian sheet: If $W_p$, $p \in \{0, 1\}^N$,
are independent $\mathbb{R}^d$-valued Gaussian white noises on $\mathbb{R}^N$ and
\begin{equation}\label{rep}
v(x):= c_H \sum_{p \in \{0, 1\}^N} \int_{\mathbb R^N} \prod_{i=1}^N \frac{f_{p_i}(x_i \xi_i)}{|\xi_i|^{H_i + 1/2}} \, W_p(d\xi),
\end{equation}
then (a continuous modification of) $\{ v(x), x \in \mathbb{R}^N_+ \}$ is an $(N, d)$-fractional Brownian
sheet with Hurst index $H$. In particular, when $H_i =\frac 1 2$ for $i = 1, \dots, k$, the Gaussian random field
$\{v(x), x \in \mathbb{R}^N\}$ is the Brownian sheet and (\ref{rep}) provides a harmonizable representation for it.

It is known \cite{WX07} that for any compact interval $I$ in $(0, \infty)^N$, there exist positive finite constants $c_1$ and $c_2$ such that for all $x, y \in I$,
\begin{equation}\label{fBs1}
c_1 \sum_{j=1}^N |x_j - y_j|^{2H_j} \le \E(|v(x) - v(y)|^2) 
\le c_2 \sum_{j=1}^N |x_j - y_j|^{2H_j}.
\end{equation}

We take $T = (0, \infty)^N$ [since $v(x) = 0$ for all $x \in \partial\R^N_+$  a.s., the existence of
multiple points is trivial on $\partial\R^N_+$].  In Lemmas \ref{lemma5.1} - \ref{lemma5.3} below, 
we use the representation (\ref{rep}) to show that
the fractional Brownian sheet satisfies the assumptions of Theorem \ref{main thm} on  $T $.

Define the random field $\{ v(A, x), A \in \mathscr{B}(\mathbb{R}_+), x \in T \}$ by
\begin{equation*}
v(A, x) = c_H \sum_{p \in \{0, 1\}^N} \int_{\{\max_i |\xi_i|^{H_i} \in A\}}
 \prod_{i=1}^N \frac{f_{p_i}(x_i \xi_i)}{|\xi_i|^{H_i + 1/2}} \, W_p(d\xi).
\end{equation*}

\begin{lemma}\label{lemma5.1}
For any $n > 1$, let $F_n = [1/n, n]^N$, $a_0 = 0$ and $\gamma_i = H_i^{-1} - 1$.
There is a constant $c_0 > 0$ depending on $n$ such that for all $0 \le a < b \le \infty$ and $x, y \in F_n$,
\begin{equation}\label{5.3}
\big\|(v(x) - v([a, b), x)) - (v(y) - v([a, b), y))\big\|_{L^2}
\le c_0 \bigg( \sum_{i=1}^N a^{\gamma_i}|x_i - y_i| + b^{-1} \bigg).
\end{equation}
\end{lemma}

\begin{proof}
Without loss of generality, we may assume $d=1$.
For any $0 \le a < b \le \infty$, let $B = \{\xi \in \R^N: \max_i |\xi_i|^{H_i} \in [a, b)\}$.
Then we can express its complement as
$$\mathbb{R}^N \setminus B = \big\{ |\xi_k| < a_k, \forall 1 \le k \le N \big\}\cup
\bigcup_{k=1}^N \big\{ |\xi_k| \ge b_k \big\},$$
where $a_i = a^{1/H_i}$ and $b_i = b^{1/H_i}$.

Note that
\[
\prod_{i=1}^N \frac{f_{p_i}(x_i \xi_i)}{|\xi_i|^{H_i+1/2}} - \prod_{i=1}^N \frac{f_{p_i}(y_i \xi_i)}{|\xi_i|^{H_i+1/2}} =
\sum_{i=1}^N \bigg( \frac{f_{p_i}(x_i\xi_i) - f_{p_i}(y_i\xi_i)}{|\xi_i|^{H_i+1/2}}
\prod_{1 \le j < i} \frac{f_{p_j}(y_j\xi_j)}{|\xi_j|^{H_j+1/2}}
\prod_{i < j \le N} \frac{f_{p_j}(x_j\xi_j)}{|\xi_j|^{H_j+1/2}} \bigg).
\]
It follows that
\begin{equation*}
\begin{split}\label{Eq:fbs1}
&  \|(v(x) - v([a, b), x)) - (v(y) - v([a, b), y))\|_{L^2}\\
& \le c_H \sum_{p \in \{0, 1\}^N} \Bigg[ \int_{\{ |\xi_k| < a_k, \forall k \}} \bigg(\prod_{i=1}^N \frac{f_{p_i}(x_i \xi_i)}{|\xi_i|^{H_i+1/2}} -
\prod_{i=1}^N \frac{f_{p_i}(y_i \xi_i)}{|\xi_i|^{H_i+1/2}} \bigg)^2  d\xi \, \Bigg]^{1/2} \\
& \quad + c_H \sum_{p \in \{0, 1\}^N} \sum_{k=1}^N \Bigg[ \int_{\{|\xi_k| \ge b_k \}} \bigg(\prod_{i=1}^N
\frac{f_{p_i}(x_i \xi_i)} {|\xi_i|^{H_i+1/2}} - \prod_{i=1}^N \frac{f_{p_i}(y_i \xi_i)}{|\xi_i|^{H_i+1/2}}
\bigg)^2  d\xi \, \Bigg]^{1/2}\\
& \le c_H \sum_{p} \sum_{i=1}^N \Bigg[ \int_{\{ |\xi_k| < a_k, \forall k \}} \bigg(\frac{f_{p_i}(x_i\xi_i) -
 f_{p_i}(y_i\xi_i)}{|\xi_i|^{H_i+1/2}} \prod_{1 \le j < i} \frac{f_{p_j}(y_j\xi_j)}{|\xi_j|^{H_j+1/2}}
 \prod_{i < j \le N} \frac{f_{p_j}(x_j\xi_j)}{|\xi_j|^{H_j+1/2}}  \bigg)^2  d\xi \, \Bigg]^{1/2} \\
& \quad + c_H \sum_{p}  \sum_{i=1}^N  \sum_{k=1}^N\Bigg[ \int_{\{|\xi_k| \ge b_k \}}
\bigg(\frac{f_{p_i}(x_i\xi_i) - f_{p_i}(y_i\xi_i)} {|\xi_i|^{H_i+1/2}} \prod_{1 \le j < i} \frac{f_{p_j}(y_j\xi_j)}
{|\xi_j|^{H_j+1/2}} \prod_{i < j \le N} \frac{f_{p_j}(x_j\xi_j)}{|\xi_j|^{H_j+1/2}}
\bigg)^2  d\xi \, \Bigg]^{1/2}.
\end{split}
\end{equation*}
Using the bounds $|f_{p_i}(x\xi) - f_{p_i}(y\xi)| \le |x - y| |\xi|$ and
$|f_{p_i}(x\xi) - f_{p_i}(y\xi)| \le 2$ for $p_i = 0$ and $1$, we see that the above is at most
\begin{align*}
& c_H \sum_{p} \sum_{i=1}^N \Bigg[ \int_{\{|\xi_i| < a_i\}} \frac{|x_i - y_i|^2}
{|\xi_i|^{2H_i - 1}} \bigg(\prod_{1 \le j < i} \frac{f_{p_j}(y_j\xi_j)}{|\xi_j|^{H_j+1/2}}
\prod_{i < j \le N} \frac{f_{p_j}(x_j\xi_j)}{|\xi_j|^{H_j+1/2}}\bigg)^2 d\xi\, \Bigg]^{1/2}\\
& + c_H \sum_{p} \sum_{i=1}^N \Bigg[ \int_{\{|\xi_i| \ge b_i\}} \frac{4}
{|\xi_i|^{2H_i+1}} \bigg(\prod_{1 \le j < i} \frac{f_{p_j}(y_j\xi_j)}{|\xi_j|^{H_j+1/2}}
\prod_{i < j \le N} \frac{f_{p_j}(x_j\xi_j)}{|\xi_j|^{H_j+1/2}}\bigg)^2 d\xi\,\Bigg]^{1/2}\\
& + c_H \sum_{p} \sum_{i=1}^N \sum_{k \ne i}  
\Bigg[ \int_{\{|\xi_k| \ge b_k\}} \frac{1}{|\xi_k|^{2H_k+1}} 
\bigg(\frac{f_{p_i}(x_i\xi_i) - f_{p_i}(y_i\xi_i)}{|\xi_i|^{H_i+1/2}} 
\prod_{j < i, j \ne k} \frac{f_{p_j}(y_j\xi_j)}{|\xi_j|^{H_j+1/2}}
\prod_{j > i, j \ne k} \frac{f_{p_j}(x_j\xi_j)}{|\xi_j|^{H_j+1/2}}\bigg)^2 d\xi\,\Bigg]^{1/2}.
\end{align*}
Then by \eqref{5.1} and \eqref{fBs1} applied on $\R^{N-1}_+$, 
the above sum is bounded from above by
\begin{align*}
& K \sum_p \sum_{i=1}^N \Bigg[ a_i^{2 - 2H_i} |x_i - y_i|^2 \prod_{1\le j < i} |y_j|^{2H_j}
\prod_{i < j \le N}|x_j|^{2H_j}  \Bigg]^{1/2}\\
& + K \sum_p \sum_{i=1}^N \Bigg[ b_i^{-2H_i} \prod_{1\le j < i} |y_j|^{2H_j}
\prod_{i < j \le N}|x_j|^{2H_j}  \Bigg]^{1/2} + K \sum_p \sum_{i=1}^N \sum_{k\ne i}\Bigg[ b_k^{-2H_k} |x_i-y_i|^{2H_i}  \Bigg]^{1/2}.
\end{align*}
Since $|x_j|, |y_j| \le n$, we obtain \eqref{5.3} for some $c_0$ depending on $n$.
\end{proof}

\begin{lemma}\label{lemma5.2}
Let $F_n = [1/n, n]^N$, $n > 1$. Take $\varepsilon_0 = 1/(2n)$, $c = 3$ and $c' = 6$.
Consider $x \in F_n$, $0 < \rho \le \varepsilon_0$ and 
a compact interval $B_{3\rho}(x) \subset F_n$.
Let $x' = x$.
There is a finite constant $C$ such that for all $y, \bar y \in B_{2\rho}(x)$, and also
for all $y, \bar{y} \in B_{2\rho}(\tilde x) \subset F_n$ with $\Delta(x, \tilde x) \ge 6\rho$,
\[
\big| \E((v_j(y) - v_j(\bar{y}))v_j(x'))\big| \le C \sum_{i=1}^N
\big|y_i - \bar{y}_i \big|^{\delta_i}
\]
for all $j \in \{1, \dots, d\}$, where $\delta_i = \min\{ 2H_i, 1 \}$.
\end{lemma}

\begin{proof}
Let $0 < \rho < \varepsilon_0$.
It suffices to show that there is a finite constant $C$ such that 
for all $y, \bar{y} \in F_n$,
\[
\bigg| \prod_{i=1}^N (|x_i|^{2H_i} + |y_i|^{2H_i} - |x_i - y_i|^{2H_i}) - \prod_{i=1}^N
(|x_i|^{2H_i} + |\bar{y}_i|^{2H_i}
- |x_i - \bar{y}_i|^{2H_i}) \bigg| \le C \sum_{i=1}^N \big|y_i - \bar{y}_i\big|^{\delta_i}.
\]
For $1 \le \ell \le N$, let $A_\ell = U_\ell - V_\ell$, where
\[
U_\ell = \prod_{i=1}^\ell \big(|x_i|^{2H_i} + |y_i|^{2H_i} - |x_i - y_i|^{2H_i}\big),
\quad V_\ell = \prod_{i=1}^\ell \big(|x_i|^{2H_i} + |\bar{y}_i|^{2H_i} - |x_i - \bar{y}_i|^{2H_i}\big).
\]
When $\ell = 1$, we have $|A_1| \le \big| |y_1|^{2H_1} - |\bar{y}_1|^{2H_1} \big| +
\big||x_1 - y_1|^{2H_1} - |x_1 - \bar{y}_1|^{2H_1}\big|$.
If $2H_1 \le 1$, then by the triangle inequality and the inequality 
$(a + b)^{2H} \le a^{2H} + b^{2H}$ for $a, b \ge 0$, we have
$|A_1| \le 2|y_1 - \bar{y}_1|^{2H_1}$;
if $2H_1 > 1$, then we can use
the mean value theorem to get $|A_1| \le 2H_1 n^{2H_1-1} |y_1 - \bar{y}_1|$. 
Thus $|A_1| \le K |y_1 - \bar{y}_1|^{\delta_1}$ for either case. 
For $2 \le \ell \le N$,
\begin{align*}
A_\ell & = U_{\ell-1}(|x_\ell|^{2H_\ell} + |y_\ell|^{2H_\ell} -
|x_\ell - y_\ell|^{2H_\ell}) - V_{\ell-1}(|x_\ell|^{2H_\ell} +
|\bar{y}_\ell|^{2H_\ell} - |x_\ell - \bar{y}_\ell|^{2H_\ell})\\
& = A_{\ell-1}(|x_\ell|^{2H_\ell} + |y_\ell|^{2H_\ell} - |x_\ell - y_\ell|^{2H_\ell})\\
& \quad + V_{\ell-1}(|y_\ell|^{2H_\ell} - |\bar{y}_\ell|^{2H_\ell} + |x_\ell - \bar{y}_\ell|^{2H_\ell} -
|x_\ell - y_\ell|^{2H_\ell}).
\end{align*}
Then $|A_\ell| \le K(|A_{\ell-1}| + |y_\ell - \bar{y}_\ell|^{\delta_\ell})$.
By induction, we deduce that 
$|A_N| \le C \sum_{\ell=1}^N |y_\ell - \bar{y}_\ell|^{\delta_\ell}$.
\end{proof}

The following lemma verifies Assumption \ref{a3} for fractional Brownian sheets.
The sectorial local nondeterminism in Theorem 1 of Wu and Xiao \cite{WX07} provides
more information on the conditional variances among $v(x_1), \dots, v(x_m)$.
\begin{lemma}\label{lemma5.3}
If $x^1, \dots, x^m \in (0, \infty)^N$ are distinct points, then the random variables
$v(x_1), \dots, v(x_m)$ are linearly independent.
\end{lemma}

\begin{proof}
Suppose that $a_1, \dots, a_m$ are real numbers such that $\sum_{\ell=1}^m a_\ell v(x^\ell) = 0$
a.s. Recalling the representation \eqref{rep} for $v(x)$, we have
\begin{align*}
0 = \E\bigg(\sum_{\ell=1}^m a_\ell v(x^\ell)\bigg)^2 =
c_H^2 \sum_{p\in \{0, 1\}^N} \int_{\R^N} \Bigg( \sum_{\ell=1}^m
a_\ell \prod_{j=1}^N \frac{f_{p_j}(x^\ell_j\xi_j)}{|\xi_j|^{H_j + 1/2}} \Bigg)^2 d\xi.
\end{align*}
Then for each $p \in \{0, 1\}^N$, $\sum_{\ell=1}^m a_\ell \prod_{j=1}^N f_{p_j}(x^\ell_j\xi_j) = 0$
and, equivalently, $\sum_{\ell=1}^m a_\ell \prod_{j=1}^N \tilde{f}_{p_j}(x^\ell_j\xi_j) = 0$
for all $\xi \in \mathbb{R}^N$,
where $\tilde{f}_0(t) = 1 - \cos{t}$ and $\tilde{f}_1(t) = -i \sin{t}$. It follows that
\begin{equation}\label{5.4}
\sum_{\ell=1}^m a_\ell \prod_{j=1}^N \left(1 - \exp({i x^\ell_j \xi_j})\right)
= \sum_{p \in \{0, 1\}^N}
 \sum_{\ell=1}^m a_\ell \prod_{j=1}^N \tilde{f}_{p_j}(x^\ell_j\xi_j)= 0
\end{equation}
for all $\xi \in \mathbb{R}^N$.
We claim that $a_1 = 0$. Let $L_{1,1}, \dots, L_{1, k_1}$ be partitions of $\{1, \dots, m\}$
obtained from the equivalence classes of the equivalence relation $\sim_1$ defined by
$\ell \sim_1 k$ if and only if $x^\ell_1 = x^k_1$. We may assume $1 \in L_{1,1}$.
Let $\hat{x}^1_1, \dots, \hat{x}^{m_1}_1$ be such that $x^\ell_1 = \hat{x}^k_1$ for all
$\ell \in L_{1,k}$, $k = 1, \dots, m_1$. Let $\xi_2, \dots, \xi_N \in \mathbb{R}$ be
arbitrary and define $c_{1,1}, c_{1,2}, \dots, c_{1,m_1}$ by
\[
c_{1,k} = \sum_{\ell \in L_{1,k}} a_\ell \prod_{j=2}^N \left( 1 - \exp({i x^\ell_j \xi_j}) \right).
\]
Then by \eqref{5.4}, we have
\[
c_{1,1} \exp(i\hat{x}^1_1 \xi_1) + \dots + c_{1,m_1} \exp(i\hat{x}^{m_1}_1 \xi_1) + (c_{1,1} + \dots + c_{1, m_1}) = 0
\]
for all $\xi_1 \in \mathbb{R}$. Since $\hat{x}^1_1, \dots, \hat{x}^{m_1}_1$ are non-zero and distinct,
the functions $\exp(i\hat{x}^1_1 \xi), \dots, \exp(i\hat{x}^{m_1}_1 \xi), 1$ are linearly independent over
$\mathbb{C}$,  we have $c_{1,1} = \dots = c_{1, m_1} = 0$.
In particular, we have
\[ \sum_{\ell \in L_{1,1}} a_\ell \prod_{j=2}^N \left( 1 - \exp({i x^\ell_j \xi_j}) \right) = 0 \]
for all $\xi_2, \dots, \xi_N \in \mathbb{R}$. Next we consider the partitions $L_{2,1},
\dots, L_{2, m_2}$ of $\{1, \dots, m\}$ obtained from equivalence classes of $\sim_2$ defined
by $\ell \sim_2 k$ iff $x^\ell_2 = x^k_2$ (with $1 \in L_{2,1}$).
Then the argument above yields
\[
\sum_{\ell \in L_{1,1} \cap L_{2,1}} a_\ell \prod_{j=3}^N \left( 1 - \exp({i x^\ell_j \xi_j}) \right) = 0.
\]
By induction, we obtain
\[
\sum_{\ell \in L_{1,1}\cap \dots \cap L_{N, 1}} a_\ell = 0.
\]
Note that $L_{1,1} \cap \dots \cap L_{N, 1} = \{1\}$ because $x^1, \dots, x^m$ are
distinct. Hence $a_1 = 0$. Similarly, we can show that $a_\ell = 0$ for  $\ell = 2, \ldots, m$.
\end{proof}

\begin{theorem}
Let $v = \{v(x), x \in \mathbb{R}^N_+ \}$ be an $(N, d)$-fractional Brownian sheet
with Hurst parameter $H \in (0, 1)^N$. If $mQ \le (m-1)d$ where $Q = \sum_{i=1}^N H_i^{-1}$,
then  $v$ has no $m$-multiple points on $(0, \infty)^N$ almost surely.
\end{theorem}

\begin{proof}
By the three lemmas above, $\{v(x),\, x \in [1/n,\, n]^N \}$ satisfies the assumptions of
Theorem \ref{main thm} with $Q = \sum_{i=1}^N H_i^{-1}$ for every $n \ge 1$. Hence the
result follows immediately from the theorem.
\end{proof}

We remark that for the case of Brownian sheet i.e. $H_i = 1/2$ for all $i$, the
above result provides an alternative proof for the main results in \cite{dknwx12,dm14} .

\medskip

\subsection{System of stochastic heat equations}

Let $k \ge 1$ and $\beta \in (0, k \wedge 2)$, or $k = 1 = \beta$. Consider the $\mathbb{R}^d$-valued
random field $\{ v(t, x), (t, x) \in \mathbb{R}_+ \times \mathbb{R}^k \}$ defined by
\[
v(t, x) = \int_{\mathbb{R}} \int_{\mathbb{R}^k} e^{-i\xi \cdot x} \frac{e^{-i\tau t} -
e^{-t|\xi|^2}}{|\xi|^2 - i\tau} |\xi|^{-(k-\beta)/2} \, W(d\tau, d\xi),
\]
where $W$ is a $\mathbb{C}^d$-valued space-time Gaussian white noise on $\mathbb{R}^{1+k}$ i.e.\
$W = W_1 + iW_2$ and $W_1, W_2$ are independent $\mathbb{R}^d$-valued space-time Gaussian white
noises on $\mathbb{R}^{1+k}$. According to Proposition 7.2 of \cite{DMX17}, the process $\hat{v}(t, x)
:= \mathrm{Re}\,{v(t, x)}$, $(t, x) \in \mathbb{R}_+ \times \mathbb{R}^k$, has the same law as the mild solution
to the system of stochastic heat equations
\begin{equation}\label{SHE}
\begin{cases}
\displaystyle{\frac{\partial}{\partial t} \hat{v}_j(t, x) = \Delta \hat{v}_j(t, x)
+ \dot{\hat{W}}_j(t, x),} & j = 1, \dots, d,\\
\hat{v}(0, x) = 0,
\end{cases}
\end{equation}
where $\hat{W}$ is an $\mathbb{R}^d$-valued spatially homogeneous Gaussian noise that is white in time with
spatial covariance $|x - y|^{-\beta}$ if $k \ge 1$ and $\beta \in (0, k \wedge 2)$; it is an $\mathbb{R}^d$-valued
space-time Gaussian white noise when $k = 1 = \beta$. Note that, in this case, we take $T = (0, \infty) \times \mathbb{R}^k$. 

By Lemma 7.3 of \cite{DMX17}, Assumption \ref{a1} is satisfied with
\[ \gamma_1 = \frac{2+\beta}{2-\beta} \quad \text{and} \quad 
\gamma_2 = \dots \gamma_{k+1} = \frac{\beta}{2-\beta}, \]
and 
\[ \alpha_1 = \frac{2-\beta}{4} \quad \text{and} \quad
\alpha_2 = \dots = \alpha_{k+1} = \frac{2-\beta}{2}. \]
In this case, 
\[\Delta((t, x), (s, y)) 
= |t - s|^{\frac{2-\beta}{4}} + |x - y|^{\frac{2-\beta}{2}}.\]

In the following lemma, we verify Assumption \ref{a2} for the
Gaussian random field $\hat{v}$.

\begin{lemma}\label{SHEa2}
Let $k \ge 1$, $\beta \in (0, k \wedge 2)$, or $k = 1 = \beta$.
Let $F \subset (0, \infty) \times \R^k$ be a compact interval.
Let $c = 3$, $c' = 6$ and $0 < \varepsilon_0 < 1$ be any small number.
For any $(t, x) \in F$ and $0 < \rho \le \varepsilon_0$ with $B_{3\rho}(t, x) \subset F$,
let $(t', x') = (t , x)$.
Take any $\delta \in (\frac{2-\beta}{2}, (2-\beta)\wedge 1)$.
Then there exists a finite constant $C$ such that 
for all $(s_1, y_1), (s_2, y_2) \in B_{2\rho}(t, x)$, and also for all
$(s_1, y_1), (s_2, y_2) \in B_{2\rho}(\tilde t, \tilde x) \subset F$ with 
$\Delta((t, x), (\tilde t, \tilde x)) \ge 6\rho$, for all $j \in \{1, \dots, d\}$,
\begin{equation}\label{Eq:SHEa2}
|\E[(\hat{v}_j(s_1, y_1) - \hat{v}_j(s_2, y_2)) \hat{v}_j(t',x')]| 
\le C (|s_1 - s_2|^{\frac{2-\beta}{2}} + |y_1 - y_2|^{\delta}).
\end{equation}
\end{lemma}

\begin{proof}
Suppose $(s_1, y_1), (s_2, y_2) \in B_{2\rho}(t, x)$ or
$(s_1, y_1), (s_2, y_2) \in B_{2\rho}(\tilde t, \tilde x) \subset F$ with 
$\Delta((t, x), (\tilde t, \tilde x)) \ge 6\rho$.
Recall that
\begin{equation*}
\hat v(t, x) = \mathrm{Re} \int_\R \int_{\R^k} e^{-i\xi \cdot x} \frac{e^{-i\tau t} -
e^{-t|\xi|^2}}{|\xi|^2 - i\tau} |\xi|^{-(k-\beta)/2} \, W(d\tau, d\xi).
\end{equation*}
Then for any $j \in \{1, \dots, d\}$,
\begin{align*}
&\quad \E[(\hat v_j(s_1, y_1) - \hat v_j(s_2, y_2)) \hat v_j(t', x')]\\
& \quad = \mathrm{Re} \int_{\R} \int_{\R^k} \bigg( e^{-i\xi \cdot y_1} \frac{e^{-i\tau s_1} -
e^{-s_1|\xi|^2}}{|\xi|^2 - i\tau} - e^{-i\xi \cdot y_2} \frac{e^{-i\tau s_2} -
e^{-s_2|\xi|^2}}{|\xi|^2 - i\tau}\bigg) \overline{\left(e^{-i\xi \cdot x'} \frac{e^{-i\tau t'} -
e^{-t'|\xi|^2}}{|\xi|^2 - i\tau}\right)} \frac{d\tau \,d\xi}{|\xi|^{k-\beta}}\\
&\quad = I_1 + I_2,
\end{align*}
where 
\begin{equation*}
I_1 = \mathrm{Re} 
\int_\R\int_{\R^k} \left(e^{-i\xi \cdot (y_1 - x')} - e^{-i\xi \cdot (y_2 - x')}\right)
\frac{(e^{-i\tau s_1} - e^{-s_1 |\xi|^2})(e^{i\tau t'} - e^{-t'|\xi|^2})}
{(|\xi|^4 + \tau^2)|\xi|^{k-\beta}} \,d\tau\,d\xi
\end{equation*}
and
\begin{equation*}
I_2 = \mathrm{Re} \int_\R \int_{\R^k} e^{-i\xi \cdot(y_2 - x')} 
\frac{(e^{-i\tau s_1} - e^{-i\tau s_2} - e^{-s_1|\xi|^2} + e^{-s_2|\xi|^2})
(e^{i\tau t'} - e^{-t'|\xi|^2})} {(|\xi|^4 + \tau^2)|\xi|^{k-\beta}} \,d\tau\,d\xi.
\end{equation*}

Let us first consider $I_1$. We can write
\begin{equation*}
e^{i\tau t'} - e^{-t'|\xi|^2} 
= (e^{i\tau t'} - 1) + (1 - e^{- t'|\xi|^2}).
\end{equation*}
It follows that $|I_1| \le J_1 + J_1'$, where
\begin{align*}
J_1 = \int_\R  \int_{\R^k} |e^{-i\xi \cdot (y_1 - x')} - e^{-i\xi \cdot (y_2 - x')}|
\frac{2|e^{i\tau t'} - 1|}{(|\xi|^4 + \tau^2) |\xi|^{k-\beta}}\, d\tau\,d\xi
\end{align*}
and 
\begin{align*}
J_1' = \int_\R \int_{\R^k} |e^{-i\xi \cdot (y_1 - x')} - e^{-i\xi \cdot (y_2 - x')}|
\frac{2(1 - e^{-t'|\xi|^2})}{(|\xi|^4 + \tau^2)|\xi|^{k-\beta}}\,d\tau\,d\xi.
\end{align*}
Choose and fix any $\delta \in (\frac{2-\beta}{2}, (2-\beta)\wedge 1)$.
We have the following elementary inequalities 
\begin{equation}\label{eq1}
|e^{iz} - 1| \le 2 \wedge |z| \le 2^{1-\delta}|z|^\delta \quad \text{for all } z \in \R,
\end{equation}
\begin{equation}\label{eq2}
|e^{-x} - e^{-y}| \le 1 \wedge |x - y| \quad  \text{for all } x, y \ge 0.
\end{equation}
For $J_1$, using these inequalities and passing to polar coordinates $r = |\xi|$,
we get that
\begin{align*}
J_1 = \int_\R d\tau \int_0^\infty dr\, 2^{1-\delta}r^\delta |y_1 - y_2|^\delta
\frac{C(1 \wedge |\tau|)}{(r^4 + \tau^2)|r^{1-\beta}},
\end{align*}
where $C$ is a constant depending on $k$ and $F$.
For fixed $\tau$, we use the change of variable $r = |\tau|^{1/2}\tilde r$ to deduce that
\begin{align*}
J_1 &\le C |y_1 - y_2|^\delta 
\int_\R d\tau\,\frac{1 \wedge |\tau|}{\tau^{1+\frac{2-\beta-\delta}{2}}}
\int_0^\infty \frac{d \tilde r}{(\tilde r^4 + 1) \tilde r^{1-\beta - \delta}}.
\end{align*}
Since the integrals in $\tau$ and $\tilde r$ are both finite, 
we have $J_1 \le C|y_1 - y_2|^\delta$, where 
the constant $C$ depends on $k, F, \beta$ and $\delta$.

For $J_1'$, we use \eqref{eq1} and \eqref{eq2}, 
pass to polar coordinates $r = |\xi|$, and then for fixed $r$, 
use the change of variable $\tau = r^2 \tilde\tau$
to get
\begin{align*}
J_1' & \le C|y_1 - y_2|^\delta \int_0^\infty 
\frac{(1 \wedge r^2)dr}{r^{3-\beta-\delta}} 
\int_\R \frac{d\tilde\tau}{1+\tilde\tau^2}.
\end{align*}
Note that the integrals in $r$ and $\tilde\tau$ are both finite.
Hence, we deduce that $|I_1| \le C |y_1 - y_2|^\delta$.

For the other integral $I_2$, we have $|I_2| \le J_2 + J_2'$, where
\begin{equation*}
J_2 = \int_\R \int_{\R^k} \frac{2|e^{-i\tau s_1} - e^{-i\tau s_2}|} {(|\xi|^4 + \tau^2)|\xi|^{k-\beta}} \,d\tau\,d\xi \quad \text{and}\quad
J_2' = \int_\R \int_{\R^k} 
\frac{2|e^{-s_1|\xi|^2} - e^{-s_2|\xi|^2}|}{(|\xi|^4 + \tau^2)|\xi|^{k-\beta}} \,d\tau\,d\xi.
\end{equation*}
For $J_2$, use the first inequality in \eqref{eq1} and use polar coordinates to get that
\begin{align*}
J_2 & \le C \int_\R d\tau \int_0^\infty dr \, \frac{2 \wedge |\tau||s_1 - s_2|}
{(r^4 + \tau^2)r^{1-\beta}}.
\end{align*}
For $\tau$ fixed, by changing variable $r = |\tau|^{1/2} \tilde r$ in the second integral,
\begin{align*}
J_2 &\le C \int_{\R} d\tau \,\frac{2 \wedge |\tau||s_1 - s_2|}{|\tau|^{2-\frac\beta 2}} 
\int_0^\infty \frac{d\tilde r}{\tilde r^4 + 1}.
\end{align*}
The integral in $\tilde r$ is finite. Then splitting the integral in $\tau$ into two parts where
$|\tau| \le \frac{2}{|s_1 - s_2|}$ and $|\tau| > \frac{2}{|s_1 - s_2|}$ leads to
\begin{align*}
J_2 &\le 
C \int_{|\tau| \le \frac{2}{|s_1 - s_2|}} \frac{|\tau||s_1 - s_2|}{|\tau|^{2-\frac\beta 2}} \,d\tau 
+ C \int_{|\tau| > \frac{2}{|s_1 - s_2|}} \frac{2}{|\tau|^{2-\frac\beta 2}} \,d\tau,
\end{align*}
which implies $J_2 \le C |s_1 - s_2|^{\frac{2-\beta}{2}}$.

For $J_2'$, by using \eqref{eq2} and polar coordinates, we have
\begin{align*}
J_2' &\le C \int_\R d\tau \int_0^\infty dr\,\frac{1 \wedge r^2|s_1 - s_2|}
{(r^4 + \tau^2)r^{1-\beta}}.
\end{align*}
Then we can permute the integrals, 
use (for $r$ fixed) the change of variable $\tau = r^2 \tilde\tau$ and 
then split the integral in $r$ into two parts where 
$r \le \frac{1}{|s_1 - s_2|^{1/2}}$ and $r > \frac{1}{|s_1 - s_2|^{1/2}}$
to derive that $J_2' \le C|s_1 - s_2|^{\frac{2-\beta}{2}}$.
This completes the proof of \eqref{Eq:SHEa2}.
\end{proof}

The next lemma verifies Assumption \ref{a3} 
and it can also be found in \cite[Lemma A.5.3]{feipu}.

\begin{lemma}\label{SHEa3}
Let $(t^1, x^1), \dots, (t^m, x^m)$ be distinct points in $(0, \infty) \times \mathbb{R}^k$.
Then the random variables $\hat{v}_1(t^1, x^1), \dots,$ $\hat{v}_1(t^m, x^m)$ are linearly independent.
\end{lemma}

\begin{proof}
Suppose that $a_1, \dots, a_m$ are real numbers such that $\sum_{j=1}^m a_j \hat{v}_1(t^j, x^j) = 0$ a.s. Then
\[
0 = \E\bigg(\sum_{j=1}^m a_j \hat{v}_1(t^j, x^j) \bigg)^2 =
\int_\mathbb{R} \int_{\mathbb{R}^k} \bigg| \sum_{j=1}^m
a_j e^{-i\xi \cdot x^j} (e^{-i\tau t^j} - e^{-t^j|\xi|^2}) \bigg|^2
\frac{d\tau\, d\xi}{(|\xi|^4+ \tau^2) |\xi|^{k - \beta}}
\]
and thus
$\sum_{j=1}^m a_j e^{-i\xi \cdot x^j} (e^{-i\tau t^j} - e^{-t^j|\xi|^2}) = 0$
for all $\tau \in \mathbb{R}$ and $\xi \in \mathbb{R}^k$.
We claim that $a_j = 0$ for all $j = 1,\dots, m$.
Let $\hat{t}^1, \dots, \hat{t}^p$ be all distinct values of the $t^j$'s.
Fix an arbitrary $\xi \in \mathbb{R}^k$. Then for all $\tau \in \mathbb{R}$, we have
\[
\sum_{\ell=1}^p \bigg( \sum_{j : t^j = \hat{t}^\ell} a_j e^{-i\xi \cdot x^j} \bigg)
e^{-i\tau \hat{t}^\ell} -
\sum_{j=1}^m a_j e^{-i\xi \cdot x^j - t^j|\xi|^2} = 0.
\]
Since the functions $e^{-i\tau \hat{t}^1}, \dots, e^{-i\tau \hat{t}^p}, 1$ are linearly independent over $\mathbb{C}$,
it follows that for all $\xi \in \mathbb{R}^k$, for all $\ell = 1, \dots, p$,
\begin{equation}\label{SHE_sum}
\sum_{j: t^j = \hat{t}^{\ell}} a_j e^{-i\xi \cdot x^j} = 0.
\end{equation}
Since $(t^1, x^1), \dots, (t^n, x^n)$ are distinct, the $x^j$'s that appear in the sum in \eqref{SHE_sum}
are distinct for any fixed $\ell$. By linear independence
of the functions $e^{-i\xi \cdot x^j}$, we conclude that $a_j = 0$ for all $j$.
\end{proof}

The following result solves the existence problem of $m$-multiple points for (\ref{SHE}).
\begin{theorem}
If $m(4+2k)/(2-\beta) \le (m-1) d$, then $\{\hat{v}(t, x), t \in (0, \infty),\, x \in \R^k\}$
has no $m$-multiple points a.s.
\end{theorem}

\begin{proof}
Assumption \ref{a1} is satisfied with $Q = (4+2k)/(2-\beta)$ by 
Lemma 7.3 of \cite{DMX17}.
Also, Assumptions \ref{a2} and \ref{a3} are satisfied by 
Lemmas \ref{SHEa2} and \ref{SHEa3} above. 
The result follows from Theorem \ref{main thm}.
\end{proof}

\bigskip

\subsection{System of stochastic wave equations}

Let $k \ge 1$ and $\beta \in (1, k \wedge 2)$ or $k = 1 = \beta$.
Consider the $\mathbb{R}^d$-valued random field $\{ v(t, x), (t, x) \in \mathbb{R}_+ \times \mathbb{R}^k \}$
defined by
\[
v(t, x) = \int_{\mathbb{R}} \int_{\mathbb{R}^k} F(t, x, \tau, \xi) |\xi|^{-(k-\beta)/2} \, W(d\tau, d\xi),
\]
where $W$ is a $\mathbb{C}^d$-valued space-time Gaussian white noise on $\mathbb{R}^{1+k}$ and
\[
F(t, x, \tau, \xi) = \frac{e^{-i\xi\cdot x - i\tau t}}{2|\xi|} \bigg[ \frac{1 - e^{it(\tau + |\xi|)}}
{\tau + |\xi|} - \frac{1 - e^{it(\tau - |\xi|)}}{\tau - |\xi|} \bigg]. \]
By Proposition 9.2 of \cite{DMX17}, the process $\hat{v}(t, x) = \mathrm{Re}\,v(t, x)$, $(t, x)
\in \mathbb{R}_+ \times
\mathbb{R}^k$, has the same law as the mild solution to the system of stochastic wave equations
\begin{equation*}
\begin{cases}
\displaystyle{\frac{\partial^2}{\partial t^2}\hat{v}_j(t, x) = \Delta \hat{v}_j(t, x) + \dot{\hat{W}}_j(t, x)}, & j = 1, \dots, d,\\
\hat{v}(0, x) = 0, \quad \displaystyle{\frac{\partial}{\partial t}\hat{v}(0, x) = 0},
\end{cases}
\end{equation*}
where $\hat{W}$ is the spatially homogeneous $\mathbb{R}^d$-valued Gaussian noise as in (\ref{SHE}).

By Lemma 9.6 of \cite{DMX17}, Assumption \ref{a1} is satisfied with
$\gamma_j = \frac{\beta}{2-\beta}$ and $\alpha_j = \frac{2-\beta}{2}$ for all $j$.
In this case, 
\[\Delta((t, x), (s, y)) 
= |t - s|^{\frac{2-\beta}{2}} + |x - y|^{\frac{2-\beta}{2}}.\]

In the following lemma, we check Assumption \ref{a2} for the 
Gaussian random field $\hat{v}$.

\begin{lemma}\label{SWEa2}
Assume that $k = 1 = \beta$ or $1 < \beta < k \wedge 2$.
Let $F \subset (0, \infty) \times \R^k$ be a compact interval.
Let $c = 4$, $c' = 8$ and $0 < \varepsilon_0 < 1$ be small enough such that 
$t - (4\varepsilon_0)^{\alpha_1^{-1}} > 0$  for all $(t, x) \in F$.
For any $(t, x) \in I$ and $0 < \rho \le \varepsilon_0$ with $B_{4\rho}(t, x) \subset F$, 
let $(t', x') = (t - (4\rho)^{\alpha_1^{-1}}, x)$. 
Then there is a finite constant $C$ such that for all 
$(s_1, y_1), (s_2, y_2) \in B_{2\rho}(t, x)$,
and also for all $(s_1, y_1), (s_2, y_2) \in B_{2\rho}(\tilde t, \tilde x) \subset F$ with
$\Delta((t, x), (\tilde t, \tilde x)) \ge 8\rho$,
for all $j \in \{1, \dots, d\}$,
\begin{equation}\label{Eq:SWEa2}
|\E[(\hat v_j(s_1, y_1) - \hat v_j(s_2, y_2)) \hat v_j(t', x')]|
\le C(|s_1 - s_2|^{2-\beta} + |y_1 - y_2|^{2-\beta}).
\end{equation}
\end{lemma}

\begin{proof}
Let $G$ be the fundamental solution of the wave equation.

First, consider the case $k = 1 = \beta$ 
(spatial dimension one with space-time white noise).
In this case, $G(s, y) = \frac{1}{2}{\bf 1}_{\{ |y| \le s \}}$.
Consider $(s_1, y_1), (s_2, y_2) \in B_{2\rho}(t, x)$ or 
$(s_1, y_1), (s_2, y_2) \in B_{2\rho}(\tilde t, \tilde x) \subset F$ with
$\Delta((t, x), (\tilde t, \tilde x)) \ge 8\rho$. 
Without loss of generality, assume $s_1 \le s_2$.

Suppose $t' \le s_1$. Then
\begin{align*}
\E[(\hat v_j(s_1, y_1) - \hat v_j(s_2, y_2))\hat v_j(t', x')]
= \frac{1}{4} \int_0^{t'} dr \int_{\R} d\bar y\, \big( {\bf 1}_{\{|\bar y - y_1| \le s_1 - r\}}
- {\bf 1}_{\{|\bar y - y_2| \le s_2 - r\}}\big) {\bf 1}_{\{ |\bar y - x'| \le t' - r \}}
\end{align*}
and
\begin{align*}
|\E[(\hat v_j(s_1, y_1) - \hat v_j(s_2, y_2))\hat v_j(t', x')]|
\le  \frac{1}{4}\int_0^{s_1} dr \int_{\R} d\bar y\, \big| {\bf 1}_{\{|\bar y - y_1| \le s_1 - r\}}
- {\bf 1}_{\{|\bar y - y_2| \le s_2 - r\}}\big|.
\end{align*}
Since the value of the integrand is either 0 or 1, it follows that
\begin{align*}
|\E[(\hat v_j(s_1, y_1) - \hat v_j(s_2, y_2))\hat v_j(t', x')]|
&\le  \frac{1}{4}\int_0^{s_1} dr \int_{\R} d\bar y\, \big| {\bf 1}_{\{|\bar y - y_1| \le s_1 - r\}}
- {\bf 1}_{\{|\bar y - y_2| \le s_2 - r\}}\big|^2\\
& \le \E|\hat v_j(s_1, y_2) - \hat v_j(s_2, y_2)|^2\\
& \le C (|s_1 - s_2| + |y_1 - y_2|).
\end{align*}
The last inequality follows from Lemma \ref{lemma1} since Assumption \ref{a1} is satisfied.

Suppose $s_1 < t'$. Then
\begin{align*}
\E[(\hat v_j(s_1, y_1) - \hat v_j(s_2, y_2))\hat v_j(t', x')]
& = \frac{1}{4} \int_0^{s_1} dr \int_{\R} d\bar y\, \big( {\bf 1}_{\{|\bar y - y_1| \le s_1 - r\}}
- {\bf 1}_{\{|\bar y - y_2| \le s_2 - r\}}\big) {\bf 1}_{\{ |\bar y - x'| \le t' - r \}}\\
& \quad - \frac{1}{4} \int_{s_1}^{s_2 \wedge t'} dr \int_{\R} d\bar y\, {\bf 1}_{\{|\bar y - y_2| \le s_2 - r\}}
{\bf 1}_{\{ |\bar y - x'| \le t' - r \}}.
\end{align*}
The first integral is bounded by $C(|s_1 - s_2| + |y_1 - y_2|)$ by the same argument 
as above, and the second integral is bounded by $C|s_1 - s_2|$ since the integrand is 
bounded by 1 and its support is also bounded over all possible $(s_2, y_2), (t', x') \in F$ 
by the compactness of $F$.
This proves \eqref{Eq:SWEa2}.

For the case $1 < \beta < k \wedge 2$ (colored noise), 
recall that $\mathcal{F}G(s, \cdot)(\xi) = \sin(s|\xi|)/|\xi|$.
If $t' \le s_1$, \eqref{Eq:SWEa2} can be proved in exactly the same way as in 
Lemma 9.6 of \cite{DMX17}.
Now suppose $s_1 < t'$. For time increments 
where $s_1 \ne s_2$ and $y_1 = y_2 = y$,
we have
\begin{align}
\begin{aligned}\label{t-inc}
&\E[(\hat v_j(s_1, y) - \hat v_j(s_2, y))\hat v_j(t', x')]\\
&= \int_0^{s_1}dr \int_{\R^k} d\xi\, |\xi|^{\beta-2-k} e^{-i\xi\cdot(y-x')} 
\Big( \sin((s_1 - r)|\xi|) - \sin((s_2 - r)|\xi|) \Big) \sin((t'-r)|\xi|)\\
& \quad - \int_{s_1}^{s_2 \wedge t'} dr \int_{\R^k} d\xi \, |\xi|^{\beta - 2 - k}
e^{-i\xi \cdot (y - x')} \sin((s_2 - r)|\xi|) \sin((t'-r)|\xi|).
\end{aligned}
\end{align}
With slight modifications of the proof of Lemma 9.6 in \cite{DMX17},
one can show that the first integral in \eqref{t-inc} is equal to
\begin{align*}
&s_1^{3-\beta} \int_{\R^k} d\eta\, |\eta|^{\beta-2-k} e^{-i\eta\cdot(y-x')} 
\left[ \Big(\cos\big(\frac{s_1 - t'}{s_1}|\eta|\big) - \cos\big(\frac{s_2 - t'}{s_1}|\eta|\big)
\Big) \right.\\
& \hspace{160pt} \left. - \frac{\sin|\eta|}{|\eta|} \Big( \cos\big(\frac{t'}{s_1}|\eta|\big) -
\cos\big( \frac{s_2 - s_1 + t'}{s_1} |\eta| \big) \Big) \right]
\end{align*}
and that it is bounded by $C|s_1 - s_2|^{2-\beta}$. 
By $|\sin(x)| \le |x| \wedge 1$, one can see that
the second integral in \eqref{t-inc} is bounded by $C|s_1 - s_2|$.
Hence \eqref{Eq:SWEa2} is satisfied for time increments.
For space increments where $s_1 = s_2$ and $y_1 \ne y_2$, 
the proof of \eqref{Eq:SWEa2}
is the same as that of Lemma 9.6 in \cite{DMX17}.
\end{proof}

\begin{lemma}\label{SWEa3}
Let $(t^1, x^1), \dots, (t^m, x^m)$ be distinct points in $ T =(0, \infty) \times \mathbb{R}^k$.
Then the random variables $\hat{v}_1(t^1, x^1), \dots,$ $\hat{v}_1(t^m, x^m)$ are linearly independent.
\end{lemma}

\begin{proof}
Suppose that $a_1, \dots, a_m$ are real numbers such that
$\sum_{j=1}^m a_j \hat{v}_1(t^j, x^j) = 0$ a.s. Then
\[
0 = \E\bigg(\sum_{j=1}^m a_j \hat{v}_1(t^j, x^j) \bigg)^2
= \int_{\mathbb{R}} \int_{\mathbb{R}^k} \bigg| \sum_{j=1}^m
a_j F(t^j, x^j, \tau, \xi) \bigg|^2 \frac{d\tau\, d\xi}{|\xi|^{k-\beta}}.
\]
It follows that $\tau \in \mathbb{R}$ and $\xi \in \mathbb{R}^k$,
$\sum_{j=1}^m a_j F(t^j, x^j, \tau, \xi) = 0$ and thus
\[ \sum_{j=1}^m b_j e^{-i\tau t^j} + c_1 \tau + c_2 = 0,
\]
where $b_j = -2 a_j |\xi| e^{-i\xi \cdot x^j}$,
$$c_1 = -\sum_{j=1}^m a_j e^{-i\xi \cdot x^j}(e^{it^j|\xi|} - e^{-it^j|\xi|})$$
and
$$c_2 = \sum_{j=1}^m a_j |\xi| e^{-i\xi \cdot x^j}(e^{it^j|\xi|} + e^{-it^j|\xi|}).
$$
We claim that $a_j = 0$ for all $j = 1, \dots, m$.
Let $\hat{t}^1, \dots, \hat{t}^p$ be all distinct values of the $t^j$'s.
If we take arbitrary $\xi \in \mathbb{R}^k$ and take derivative with respect to $\tau$, we see that
\[
\sum_{\ell = 1}^p \bigg(-i\hat{t}^\ell \sum_{j: t^j = \hat{t}^\ell} b_j\bigg)e^{-i\tau \hat{t}^\ell} + c_1
= 0
\]
for all $\tau \in \mathbb{R}$. Since the functions $e^{-i\tau \hat{t}^1}, \dots,
e^{-i\tau \hat{t}^p}, 1$ are linearly independent over $\mathbb{C}$, we have
\[ -i\hat{t}^1 \sum_{j: t^j = \hat{t}^\ell} b_j= 0 \]
for all $\ell = 1, \dots, p$.
It implies that for all $\xi \in \mathbb{R}^k$, for all $\ell = 1, \dots, p$,
\begin{equation}\label{SWE_sum}
 \sum_{j: t^j = \hat{t}^\ell} a_j e^{-i\xi\cdot x^j} = 0.
\end{equation}
Since $(t^1, x^1), \dots, (t^m, x^m)$ are distinct, the $x^j$'s that appear in the
sum in \eqref{SWE_sum} are distinct for any fixed $\ell$. By linear independence
of the functions $e^{-i\xi\cdot x^j}$, we conclude that $a_j = 0$ for all $j$.
\end{proof}

\begin{theorem}
Assume $k = 1 =\beta$ or $1 < \beta < k \wedge 2$.
If $m(2+2k)/(2-\beta) \le (m-1)d$, then $\{\hat{v}(t, x), t \in (0, \infty),\, x \in \R^k\}$
has no $m$-multiple points a.s.
\end{theorem}

\begin{proof}
Assumption \ref{a1} is satisfied with $Q = (2+2k)/(2-\beta)$ 
by Lemma 9.3 of \cite{DMX17}. 
Assumptions \ref{a2} and \ref{a3} are satisfied by 
Lemmas \ref{SWEa2} and \ref{SWEa3} above. 
Hence the result follows from Theorem \ref{main thm}.
\end{proof}




\bigskip

\begin{thebibliography}{99}

\bibitem{Anderson}
T. W.  Anderson,  The integral of a symmetric unimodal function over a symmetric convex
set and some probability inequalities. {\it Proc. Amer. Math. Soc.} {\bf 6} (1955), 170--176.

\bibitem{DKN07}
R. C. Dalang, D. Khoshnevisan and E. Nualart,  Hitting probabilities for systems of
non-linear stochastic heat equations with additive noise.  \textit{Latin Amer. J. Probab.
Statist. (ALEA)} \textbf{3} (2007), 231--271.

\bibitem{dknwx12}
R. C. Dalang,  D. Khoshnevisan, E. Nualart, D. Wu and Y.  Xiao,  Critical {B}rownian sheet
 does not have double points. {\em Ann. Probab.} {\bf 40} (2012), no.~4, 1829--1859.

\bibitem{dm14}
R. C. Dalang and  C. Mueller,  Multiple points of the Brownian sheet in critical
dimensions. {\em Ann. Probab.} {\bf 43} (2015), no.~4, 1577--1593.


\bibitem{DMX17}
R. C. Dalang, C. Mueller and Y. Xiao, Polarity of points for Gaussian random fields.
{\it Ann. Probab.} {\bf 45} (2017), no.\ 6B, 4700--4751.

\bibitem {G81}
A. Goldman, Points multiples des trajectoires de processus Gaussiens.
 {\it Z. Wahrsch. Verw. Gebiete.} {\bf 75} (1981), 481--494.

\bibitem{KRS12}
A. K\"{a}enm\"{a}ki, T. Rajala and V. Suomala, Existence of doubling measures via
generalised nested cubes. {\it Proc. Amer. Math. Soc.} {\bf 140} (2012), no. 9, 3275--3281.

\bibitem {Kono}
N. K\^ono, Double points of a Gaussian sample path. {\it Z. Wahrsch. Verw. Gebiete.}
{\bf 45} (1978), 175--180.

\bibitem{Ledoux}
M. Ledoux, Isoperimetry and Gaussian analysis. {\it Lectures on Probability Theory
and Statistics (Saint-Flour, 1994). Lecture Notes in Math.} {\bf 1648} 165--294.
Springer, Berlin, 1996.

\bibitem{MT}
C. Mueller and R. Tribe, Hitting properties of a random string. {\em Electron. J. Probab.}
{\bf 7} (2002), no. 10, 29 pp.

\bibitem{feipu} F. Pu, The stochastic heat equation: hitting probabilities and the probability 
density function of the supremum via Malliavin calculus. Ph.D. thesis no.~8695 (2018), 
Ecole Polytechnique F\'ed\'erale de Lausanne, Switzerland.

\bibitem{Rosen84}
J. Rosen, Self-intersections of random fields. {\it Ann. Probab.} {\bf 12} (1984), 108--119.

\bibitem{T98}
M. Talagrand, Multiple points of trajectories of multiparameter fractional Brownian motion.
{\it Probab. Theory Relat. Fields} {\bf 112} (1998), 545--563.

\bibitem{T95}
M. Talagrand, Hausdorff measure of trajectories of multiparameter fractional Brownian motion.
{\em Ann. Probab.} {\bf 23} (1995), no.~2, 767--775.

\bibitem{WX07}
D. Wu and Y. Xiao, Geometric properties of fractional Brownian sheets. {\it J.
Fourier Anal. Appl.} {\bf 13} (2007), 1--37.


\bibitem{xiao09}
Y. Xiao, Sample path properties of anisotropic Gaussian random fields. { \em A minicourse on
stochastic partial differential equations,} (D  Khoshnevisan and F  Rassoul-Agha, editors),
{\em Lecture Notes in Math}, \textbf{1962},   Springer, New York (2009), 145--212

\bigskip
\bigskip

\end{thebibliography}
\end{document}